\documentclass[12pt, reqno]{amsart}
\usepackage[utf8]{inputenc}
\usepackage{graphicx}
\usepackage{amsmath}
\usepackage{amsthm}
\usepackage{amssymb,bbm}%
\usepackage[numbers, square]{natbib}

\theoremstyle{plain}
\newtheorem{theorem}{Theorem}[section]
\newtheorem{lemma}[theorem]{Lemma}
\newtheorem{corollary}[theorem]{Corollary}
\newtheorem{proposition}[theorem]{Proposition}

\theoremstyle{definition}

\theoremstyle{remark}
\newtheorem{remark}[theorem]{Remark}

\makeindex
\makeglossary

\newcommand{\cF}{\mathcal{F}}

\newcommand{\cQ}{\mathcal{Q}}

\newcommand{\E}{\mathbb E\,}
\newcommand{\R}{\mathbb{R}}
\newcommand{\N}{\mathbb{N}}

\renewcommand{\P}{\mathbb{P}}

\newcommand{\bd}{\mathop{\mathrm{bd}}\nolimits}

\newcommand{\conv}{{\mathrm{conv}}}

\newcommand{\aff}{\mathop{\mathrm{aff}}\nolimits}
\newcommand{\dist}{\mathop{\mathrm{dist}}\nolimits}

\newcommand{\Span}{\mathop{\mathrm{span}}\nolimits}

\newcommand{\eqdistr}{\stackrel{d}{=}}

\newcommand{\ind}{\mathbbm{1}}

\newcommand{\bx}{\mathbf{x}}
\newcommand{\by}{\mathbf{y}}
\newcommand{\bu}{\mathbf{u}}

\newcommand{\dd}{{\,\rm d}}

\newcommand{\given}{\,|\,}
\newcommand{\bgiven}{\,\big|\,}

\newcommand{\BS}{\mathbf{S}}
\renewcommand{\cQ}{\mathcal{C}}

\usepackage{hyperref}

\makeindex
\makeglossary  

\begin{document}

\author[J.~Randon-Furling]{Julien Randon-Furling}
\address{Julien Randon-Furling, Universit\'e Paris 1 Panth\'eon Sorbonne, SAMM -- CNRS, FP2M (CNRS FR2036), 90 rue de Tolbiac, 75013 Paris, France \&~Columbia University, Department of Mathematics, 2990 Broadway, New York, NY 10027, USA}
\email{Julien.Randon-Furling@cantab.net}

\author[D.~Zaporozhets]{Dmitry Zaporozhets}
\address{Dmitry Zaporozhets, St.~Petersburg Department of the Steklov Institute of Mathematics, Fontanka~27, 191011 St.~Petersburg, Russia}
\email{zap1979@gmail.com}

\title[Convex hulls of several random walks]{Convex hulls of several\\ multidimensional  Gaussian random walks}

\begin{abstract}
We derive explicit formulae for the expected volume and the expected number of facets of the convex hull of several multidimensional Gaussian random walks in terms of the  Gaussian persistence probabilities. Special cases include the already known results about the convex hull of a single Gaussian random walk and the $d$-dimensional Gaussian polytope with or without the origin.
\end{abstract}

\maketitle

\section{Introduction}
\subsection{Random walk in $\R^1$}
Consider a one-dimensional random walk
\begin{align}\label{1648}
    S_i=X_1+\dots+X_i,\quad i=1,\dots,n,
\end{align}
where  $X_1,\dots, X_n$ are i.i.d. random variables. The classical result of Sparre Andersen~\cite{eSA49} states that if the steps are symmetrically and absolutely continuously distributed, then the probability for the random walk to stay positive (\emph{the persistence probability}) is distribution-free and given by
\begin{align}\label{1141}
    \P[S_1\geq 0,\dots S_n\geq0]=\frac{(2n-1)!!}{(2n)!!}.
\end{align}

Another very well-known result, also due to Sparre Andersen~\cite{eSA54}, under the same assumptions calculates the distribution of the random walk's maximum position (\emph{the discrete arcsine law}): for $i=0,\dots,n$,
\begin{align}\label{1320}
    \P[\max(S_0,\dots,S_n)=S_i]=\frac{(2i-1)!!}{(2i)!!}\frac{(2n-2i-1)!!}{(2n-2i)!!}.
\end{align}
Summing up over $i$ gives the following version of the Chu--Vandermonde identity:
\begin{align}\label{2011}
    \sum_{i=0}^n\frac{(2i-1)!!}{(2i)!!}\frac{(2n-2i-1)!!}{(2n-2i)!!}=1.
\end{align}

\subsection{Random walk in $\R^d$.}
Because of the symmetry, Equations~\eqref{1141} and~\eqref{1320} are equivalent to
\begin{align}\label{1327}
    \P[0\not\in \conv(S_1,\dots S_n)]=2\,\frac{(2n-1)!!}{(2n)!!}
\end{align}
and
\begin{align}\label{1331}
    \P[S_i\,\,\text{is an edge of}\,\,\conv(S_0,\dots,S_n)]=2\,\frac{(2i-1)!!}{(2i)!!}\frac{(2n-2i-1)!!}{(2n-2i)!!},
\end{align}
where in the one-dimensional case 
$$
\conv(S_1,\dots S_n)=\left\lbrace\min(S_1,\dots,S_n), \max(S_1,\dots,S_n)\right\rbrace,
$$
and an ``edge'' reduces to a point (either $\min$ or $\max$).
In this form, the formulae can be naturally generalized to higher dimensions. Again let $S_0, S_1,\dots,S_n$ be  a  random walk  defined as in~\eqref{1648}, where  the steps  are \emph{$d$-dimensional } i.i.d. random vectors now:
\begin{align*}
    X_1,\dots,X_n\in\R^d.
\end{align*}
As in the one-dimensional case, we assume that they are symmetrically and absolutely continuously distributed. Generalizing~\eqref{1327} it was shown in~\cite{KVZ17a} that

\begin{equation} \label{125}
\P[0 \notin \conv(S_1,S_2,\ldots,S_n)] = 2\,\frac{P^{(n)}_{d-1}+P^{(n)}_{d-3}+\dots}{(2n)!!},
\end{equation}
where $P^{(n)}_{j}$ are the coefficients of the polynomial
$$
(t+1)(t+3)\ldots (t+2n-1) = \sum_{j=0}^n P^{(n)}_{j} t^j.
$$
The left-hand side of~\eqref{125} is often referred to as \emph{the non-absorption probability}.

Equation~\eqref{1331} was generalized in~\cite{VZ18} as follows: let $0\leq i_1<\dots<i_d\leq n$ be any indices. Then,
\begin{align*}
    \P[\conv(S_{i_1},\dots,S_{i_d})\,\,&\text{is a facet of}\,\,\conv(S_0,\dots,S_n)]
    \\
    &=2\,\frac{(2i_1-1)!!}{(2i_1)!!}\frac{(2n-2i_d-1)!!}{(2n-2i_d)!!}\prod_{j=1}^{d-1}\frac{1}{i_{j+1}-i_j}.
\end{align*}
This expression  is naturally called \emph{the facet probability}.  Later, in~\cite{KVZ17b} this result was extended  to faces of any dimension as follows. For  $k=0,1,\dots,d-1$, denote by $\cF_k(\cdot)$ the set of $k$-dimensional faces of a polytope. Then for any indices 
\begin{align}\label{2012}
   0 \leq i_1<\dots<i_{k+1}\leq n
\end{align}
we have
\begin{align*}
    \P[\conv(S_{i_1},\dots,&S_{i_{k+1}})\in\cF_k(\conv(S_0,\dots,S_n))]
    \\
    &=2\,\frac{P^{(n,i_1,\dots,i_{k+1})}_{d-k-1}+P^{(n,i_1,\dots,i_{k+1})}_{d-k-3}+\dots}{(2i_1)!!(2i_2-2i_1)!!\dots (2i_{k+1}-2i_k)!!(2n-2i_{k+1})!!},
\end{align*}
where $P_{j}^{(n,i_1,\dots,i_{k+1})}$ are the coefficients of the polynomial
\begin{align*}
(t+1)(t+3)\ldots (t+2i_1-1)\times (t+1)(t+3)\ldots (t+2n-2i_{k+1}-1)
\\
\times \prod_{l=1}^k[(t+1)(t+2)\dots(t+i_{l+1}-i_l-1)]
= \sum_{j=0}^n P^{(n,i_1,\dots,i_{k+1})}_{j} t^j.
\end{align*}
Summing this up over all $(k+1)$-tuples from~\eqref{2012} gives the average number of $k$-faces of the convex hull:
\begin{align*}
   \E&|\cF_k(\conv(S_0,\dots,S_n))|
    \\
    &=2\, \sum_{0 \leq i_1<\dots<i_{k+1}\leq n}\frac{P^{(n,i_1,\dots,i_{k+1})}_{d-k-1}+P^{(n,i_1,\dots,i_{k+1})}_{d-k-3}+\dots}{(2i_1)!!(2i_2-2i_1)!!\dots (2i_{k+1}-2i_k)!!(2n-2i_{k+1})!!}.
\end{align*}

Surprisingly, as was shown in~\cite{KVZ17b}, this formula remains true even \emph{without} symmetry assumption for the step distribution.

If $k=d-1$, then the formula reduces to the formula of Barndorff-Nielsen and Baxter~\cite{BB63} (see also~\cite{VZ18}):
\begin{align}\label{2238}
    \E|\cF_{d-1}(\conv(S_0,\dots,S_n))|=
    2\sum_{\substack{j_1+\dots+j_{d-1}\leq n \\ j_1,\dots,j_{d-1}\geq1}}
	 \frac1{j_1\dots j_{d-1}}.
\end{align}

\subsection{Several random walks}\label{1205}
Now let us turn to the case of \emph{several} random walks which is our main interest in this paper.

Fix $m,n_1,\dots,n_m \in \N$,  and let
\begin{align*}
    X_1^{(1)}, \dots, X_{n_1}^{(1)}, \;\; \dots,\;\; X_1^{(m)}, \dots, X_{n_m}^{(m)}
\end{align*}
be independent $d$-dimensional random vectors. As above, we assume that they are symmetrically and absolutely continuously distributed and for any $l=1,\dots,m$, the vectors 
\begin{align*}
    X_1^{(l)}, \dots, X_{n_l}^{(l)}\quad\text{are identically distributed.}
\end{align*}

Consider the collection of $m$ random walks $(S_i^{(1)})_{i=1}^{n_1},\dots, (S_i^{(m)})_{i=1}^{n_m}$ defined as
\begin{align*}
S_i^{(l)} = X_1^{(l)} + \dots + X_{i}^{(l)}, \, 1\leq i \leq n_l,\,  1\leq l \leq m.
\end{align*}
We aim to study the properties of their convex hull 
\begin{align*}
    \cQ_{d}:=\conv \left(S_1^{(1)}, \dots, S_{n_1}^{(1)}, \;\; \dots,\;\; S_1^{(m)}, \dots, S_{n_m}^{(m)}\right)
\end{align*}
and also the convex hull with the origin
\begin{align*}
    \cQ_{d}^0:=\conv \left(0,S_1^{(1)}, \dots, S_{n_1}^{(1)}, \;\; \dots,\;\; S_1^{(m)}, \dots, S_{n_m}^{(m)}\right).
\end{align*}

Under these quite general assumptions, it was shown in~\cite{KVZ17a} that
\begin{equation*}
\P[0 \notin \cQ_d] = 2\,\frac{P^{(n_1,\dots,n_m)}_{d-1}+P^{(n_1,\dots,n_m)}_{d-3}+\dots}{(2n)!!},
\end{equation*}
where $P^{(n_1,\dots,n_m)}_{k}$ are the coefficients of the polynomial
$$
\prod_{l=1}^m\big((t+1)(t+3)\ldots (t+2n_l-1)\big) = \sum_{k=0}^{n_1+\dots+n_m} P^{(n_1,\dots,n_m)}_{k} t^k.
$$
Taking $n_1=\dots=n_m=1$ recovers the classical Wendel formula~\cite{jW62}
\begin{align*}
    \P[0\notin\conv(X^{(1)},X^{(2)}\dots,X^{(m)})]=\frac1{2^{m-1}}\sum_{j=0}^{d-1}\binom{m-1}j,
\end{align*}
where $X^{(1)}:=X_1^{(1)},\dots, X^{(m)}:=X_1^{(m)}$. Again, let us stress that this non-absorption probability is distribution-free.

Yet, the  average number of facets of $\conv(X^{(1)},X^{(2)}\dots,X^{(m)})$ \emph{does depend} on the step distribution, see~\cite{hC70},~\cite{fA91},~\cite{rD91},~\cite{eH11},~\cite{KTT19}, where the asymptotic formulae for the different distribution classes were obtained.

Thus the average number of facets of $\cQ_{d}$ and $\cQ_{d}^0$ \emph{is no longer distribution free}. In this paper, we study the case when the steps are \emph{the standard Gaussian vectors}.

\subsection{Paper structure}
In the next section, we fix basic notation and formulate our main theorem which gives the   expected value of some general geometrical functional of $\cQ_{d}$, $\cQ_{d}^0$.

In Section~\ref{825}, from this main theorem we derive the formulae for the expected number of facets, volume and surface area of $\cQ_{d}$, $\cQ_{d}^0$.

As examples, several known results are derived in Section~\ref{2129}.

In our proofs, we will be faced with having to calculate the first and second moments of some random determinant of Gaussian type. To this end, in Section~\ref{1138} we formulate a general result that  gives all positive moments of the volume of a random simplex from some class of Gaussian simplices. 

The proofs are located in Section~\ref{2139}.



\section{Setting and Main Results}
In the remainder of the paper we assume that the steps
\begin{align*}
    X_1^{(1)}, \dots, X_{n_1}^{(1)}, \;\; \dots,\;\; X_1^{(m)}, \dots, X_{n_m}^{(m)}\in\R^d
\end{align*}
are independent standard Gaussian random vectors.

With probability~$1$, ${\cQ_{d}}$ and ${\cQ_{d}^0}$ are convex polytopes with boundaries of the form
\begin{align*}
    \partial\,{\cQ_{d}}=\bigcup_{F\in \cF({\cQ_{d}})} F\quad\text{and}\quad \partial\,{\cQ_{d}^0}=\bigcup_{F\in \cF({\cQ_{d}^0})} F,
\end{align*}
where $\cF(\cdot):=\cF_{d-1}(\cdot)$ stands for the set of facets ($(d-$1$)$-dimensional faces) of a polytope. Each facet is a $(d-$1$)$-dimensional simplex almost surely.

Fix some non-negative integers 
\begin{align}\label{1207}
    k_1,\dots,k_m\quad\text{such that}\quad k_l\leq n_l\quad\text{and}\quad\sum_{l=1}^mk_l=d,
\end{align}
then fix $d$ indices
\begin{align}\label{1208}
    1\leq i_1^{(l)}<\dots< i_{k_l}^{(l)}\leq n_l\quad\text{for} \quad l=1,\dots,m \quad\text{such that}\quad k_l>0,
\end{align}
and denote by $\BS_{d}$ a $d$-tuple defined as
\begin{align}\label{1212}
    \BS_{d}:=\bigg(S_{i_1^{(1)}}^{(1)},\dots, S_{i_{k_1}^{(1)}}^{(1)}, \dots,  S_{i_1^{(m)}}^{(m)}, \dots, S_{i_{k_m}^{(m)}}^{(m)}\bigg)
\end{align}
with the convention that
\begin{align*}
    \left\{S_{i_1^{(l)}}^{(l)}, \dots, S_{i_{k_l}^{(l)}}^{(l)}\right\}:=\emptyset\quad\text{for}\quad k_l=0.
\end{align*}
We also write
\begin{align*}
    \conv\BS_{d}:=\conv\left(S_{i_1^{(1)}}^{(1)}, \dots, S_{i_{k_1}^{(1)}}^{(1)}, \dots,  S_{i_1^{(m)}}^{(m)}, \dots, S_{i_{k_m}^{(m)}}^{(m)}\right).
\end{align*}

In the same way fix some non-negative integers 
\begin{align}\label{01207}
    k_1,\dots,k_m\quad\text{such that}\quad k_l\leq n_l\quad\text{and}\quad\sum_{l=1}^mk_l=d-1,
\end{align}
and fix $d-1$ indices
\begin{align}\label{01208}
    1\leq i_1^{(l)}<\dots< i_{k_l}^{(l)}\leq n_l\quad\text{for} \quad l=1,\dots,m \quad\text{such that}\quad k_l>0,
\end{align}
and similarly let
\begin{align}\label{01212}
    \BS_{d}^0:=\bigg(0,S_{i_1^{(1)}}^{(1)},\dots, S_{i_{k_1}^{(1)}}^{(1)}, \dots,  S_{i_1^{(m)}}^{(m)}, \dots, S_{i_{k_m}^{(m)}}^{(m)}\bigg)
\end{align}
and
\begin{align*}
    \conv\BS_{d}^0:=\conv\left(0,S_{i_1^{(1)}}^{(1)}, \dots, S_{i_{k_1}^{(1)}}^{(1)}, \dots,  S_{i_1^{(m)}}^{(m)}, \dots, S_{i_{k_m}^{(m)}}^{(m)}\right).
\end{align*}


Note that $\conv\BS_{d}$ may or may not be a facet of ${\cQ_{d}}$. Moreover, every facet $F\in\cF({\cQ_{d}})$ can be represented as $\conv\BS_{d}$ with $\BS_{d}$ defined in~\eqref{1212} with some integers~\eqref{1207} and indices~\eqref{1208}. Therefore, we arrive at the following elementary albeit  crucial relation: with probability one,
\begin{align}\label{1151}
    \sum_{F\in \cF({\cQ_{d}})}g(F)=\sum_{\substack{k_1+\dots+k_m=d \\ 0\leq k_l\leq n_l, \,l=1,\dots,m}}\quad \sum_{\substack{1\leq i_1^{(l)}<\dots< i_{k_l}^{(l)}\leq n_l \\ l=1,\dots,m\,:\, k_l>0}}g(\BS_{d})\ind_{\lbrace\BS_{d}\in\cF({\cQ_{d}})\rbrace},
\end{align}
where $g:(\R^d)^d \to \R^1$ is an arbitrary symmetric, non-negative, measurable function. Here, we write
\begin{align*}
    g(F):=g(\bx_1^F,\dots,\bx_d^F),\quad\text{where $\bx_1^F,\dots,\bx_d^F$ are the vertices of $F$.}
\end{align*}
For the convex hull with the origin, the analogous relation is a little bit more complicated: with probability one,
\begin{align}\label{01151}
    \sum_{F\in \cF({\cQ_{d}^0})}g(F)=&\sum_{\substack{k_1+\dots+k_m=d \\ 0\leq k_l\leq n_l, \,l=1,\dots,m}}\quad \sum_{\substack{1\leq i_1^{(l)}<\dots< i_{k_l}^{(l)}\leq n_l \\ l=1,\dots,m\,:\, k_l>0}}g(\BS_{d})\ind_{\lbrace\BS_{d}\in\cF({\cQ_{d}^0})\rbrace}
    \\\notag&+\sum_{\substack{k_1+\dots+k_m=d-1 \\ 0\leq k_l\leq n_l, \,l=1,\dots,m}}\quad \sum_{\substack{1\leq i_1^{(l)}<\dots< i_{k_l}^{(l)}\leq n_l \\ l=1,\dots,m\,:\, k_l>0}}g(\BS_{d}^0)\ind_{\lbrace\BS_{d}^0\in\cF({\cQ_{d}^0})\rbrace}.
\end{align}

We aim to calculate the expectations in the right-hand sides in~\eqref{1151} and~\eqref{01151}, which by taking $g\equiv 1$ will readily give us the expected number of facets of ${\cQ_{d}}$ and ${\cQ_{d}}^0$. Moreover, by making a more subtle choice of $g$, we will be able to find the expected volumes and surface areas of ${\cQ_{d}}$ and ${\cQ_{d}^0}$.


Our results will be expressed in terms of the unconditional and conditional Gaussian persistence probabilities: for $r\in\R^1$ let
\begin{align}\label{1350}
    p_n(r)&:=\P\big[\sum_{i=1}^kN_i\leq r,k=1,\dots,n\big],
    \\\notag
    q_n(r)&:=\P\big[\sum_{i=1}^kN_i\leq r,k=1,\dots,n\given \sum_{i=1}^nN_i=r\big],
\end{align}
where $N_1,\dots, N_n$ are independent standard Gaussian random variables. Due to the symmetry of the distribution, for any $r\geq 0$ the definition of $q_n(r)$ is equivalent to
\begin{align}\label{1907}
    q_n(r)&:=\P\big[\sum_{i=1}^kN_i\geq 0,k=1,\dots,n-1\given \sum_{i=1}^nN_i=r\big].
\end{align}
Note that
\begin{align}\label{3}
    p_1(r)=\Phi(r),\quad q_1(r)\equiv 1,
\end{align}
and \begin{align}\label{1900}
p_n(0)=\frac{(2n-1)!!}{(2n)!!},\quad q_n(0)=\frac1n,
\end{align}
where $\Phi(r)=\int_{-\infty}^{r}\varphi(s)\dd s$ is the cumulative distribution of the standard Gaussian density $\varphi(r)$, and~\eqref{1900} was established by Sparre Andersen~\cite{eSA49,eSA53} (who proved it in a setting much more general than the Gaussian one considered here; the first part of~\eqref{1900} was mentioned in the introduction, see~\eqref{1141}).
\medskip

For a linear hyperplane $L\subset\R^d$ denote by $P_L$ the operator of the orthogonal projection onto $L$. In particular, denote by $P_d$ the operator of the orthogonal projection onto the first $d-$1 coordinates:
\begin{align*}
    P_d\,:\,\bx=(x_1,\dots,x_{d-1},x_d)\mapsto (x_1,\dots,x_{d-1}, 0).
\end{align*}

The $d$-dimensional volume is denoted by $|\cdot|$. Some of the sets we consider have dimension $d-$1, for them $|\cdot|$ stands for the $(d-$1$)$-dimensional volume. For finite sets, $|\cdot|$ denotes cardinality.

Let $\mathbb S^{d-1}\subset\R^d$ denote the unit $(d-$1$)$-dimensional sphere centered at the origin and equipped with the Lebesgue measure $\mu$ normalized to be probabilistic. For $u\in\mathbb S^{d-1}$ denote by $u^\perp$ the linear hyperplane orthogonal to $u$.

Now we are  in position  to formulate our main result.

\begin{theorem}\label{2047}
    Let $g:(\R^d)^d \to \R^1$ be some bounded measurable function invariant with respect to translations and rotations.\footnote{By this we mean that $g(\mathbf a+Q\bx_1,\dots,\mathbf a+Q\bx_d)=g(\bx_1,\dots,\bx_d)$ for any $\mathbf a,\bx_1,\dots,\bx_d\in\R^d$ and any orhtogonal matrix $Q\in\R^{d\times d}$. } Consider some $\BS_{d}$ defined in~\eqref{1212} with some integers from~\eqref{1207} and indices from~\eqref{1208}. Then,
\begin{align}\label{2140}
	 \E&[g(\BS_{d})\ind_{\lbrace\conv\BS_{d}\in\cF({\cQ_{d}})\rbrace}]=\notag  \frac{2^{d/2}\Gamma\big(\frac {d+1}2\big)}{\sqrt\pi}\,\E\big[g( P_d\BS_{d})\, |\mathrm{conv} P_d\BS_{d}|\big]
	 \\
	 &\times\prod_{l\,:\,k_l \ne0}\Bigg[\Big({i_1^{(l)}}\Big)^{-1/2}\Big(\big(i_2^{(l)}-i_1^{(l)}\big)\dots\big(i_{k_l}^{(l)}-i_{k_{l}-1}^{(l)}\big)\Big)^{-3/2}\cdot\frac{(2(n_l-i^{(l)}_{k_l})-1)!!}{(2(n_l-i^{(l)}_{k_l}))!!}\Bigg]\notag
	 \\&\times\int_{-\infty}^\infty \exp\bigg(-\frac {r^2}{2}\sum_{l\,:\, k_l\ne0}\frac{1}{i_1^{(l)}}\bigg)\prod_{l\,:\,k_l =0}p_{n_l}(r)\prod_{l\,:\,k_l \ne0}q_{i_1^{(l)}}(r)  \dd r
\end{align}    
and
\begin{align}\label{3140}
	 \E&[g(\BS_{d})\ind_{\lbrace\conv\BS_{d}\in\cF({\cQ_{d}^0})\rbrace}]=\notag   \frac{2^{d/2}\Gamma\big(\frac {d+1}2\big)}{\sqrt\pi}\,\E\big[g( P_d\BS_{d})\, |\mathrm{conv} P_d\BS_{d}|\big]
	 \\
	 &\times\prod_{l\,:\,k_l \ne0}\Bigg[\Big({i_1^{(l)}}\Big)^{-1/2}\Big(\big(i_2^{(l)}-i_1^{(l)}\big)\dots\big(i_{k_l}^{(l)}-i_{k_{l}-1}^{(l)}\big)\Big)^{-3/2}\cdot\frac{(2(n_l-i^{(l)}_{k_l})-1)!!}{(2(n_l-i^{(l)}_{k_l}))!!}\Bigg]\notag
	 \\&\times\int_0^\infty \exp\bigg(-\frac {r^2}{2}\sum_{l\,:\, k_l\ne0}\frac{1}{i_1^{(l)}}\bigg)\prod_{l\,:\,k_l =0}p_{n_l}(r)\prod_{l\,:\,k_l \ne0}q_{i_1^{(l)}}(r)  \dd r.
\end{align}    

Also consider some $\BS_{d}^0$ defined in~\eqref{01212} with some integers from~\eqref{01207} and indices from~\eqref{01208}. Then,  
\begin{align}\label{1152}
	 \notag\E&[g(\BS_{d}^0)\ind_{\lbrace\conv\BS_{d}^0\in\cF({\cQ_{d}^0})\rbrace}]=
	 2^{\frac{d+1}{2}}\Gamma\Big(\frac {d+1}2\Big)\,\E\big[g( P_d\BS_{d}^0)\, |\mathrm{conv} P_d\BS_{d}^0|\big]
	 \\
	 &\times\prod_{l=1}^m\Bigg[\Big(i_1^{(l)}\big(i_2^{(l)}-i_1^{(l)}\big)\dots\big(i_{k_l}^{(l)}-i_{k_{l}-1}^{(l)}\big)\Big)^{-3/2}\cdot\frac{(2(n_l-i^{(l)}_{k_l})-1)!!}{(2(n_l-i^{(l)}_{k_l}))!!}\Bigg],
\end{align}    
with the convention that for $k_l=0$,  the factor in the product equals ${(2n_l-1)!!}/{(2n_l)!!}$. 
\end{theorem}
\noindent The proof is based on an idea from~\cite{RF12} and given in Section~\ref{2139}. Now let us derive some applications.

\section{Applications}\label{825}
\subsection{Moments of Gaussian simplex volume}
In the right-hand sides of~\eqref{2140},~\eqref{3140}, and~\eqref{1152}, the terms
\begin{align*}
    \E\big[g( P_d\BS_{d})\, |\mathrm{conv} P_d\BS_{d}|\big],\quad\E\big[g( P_d\BS_{d}^0)\, |\mathrm{conv} P_d\BS_{d}^0|\big]
\end{align*}
are difficult to deal with. To get rid of them in the forthcoming applications, we will need formulae for the moments of $|\conv P_d\BS_{d}|,|\conv P_d\BS_{d}^0|$. In Section~\ref{1138}, we will obtain a much more general result of this type: we will find  all positive moments of a Gaussian simplex generated by several independent Gaussian walks with arbitrarily weighted steps. Now we formulate  only the statement which is necessary for us.
\begin{proposition}\label{1157}
For any $p>0$,
\begin{align*}
    \E&|\conv P_d\BS_{d}|^p=
    \prod_{i=1}^{d-1}\frac{\Gamma\big(\frac{i+p}{2}\big)}{\Gamma\big(\frac{i}{2}\big)}
    \\
    &\times\Bigg[\frac{2^{d-1}}{((d-1)!)^2}\bigg(\sum_{l:k_l\ne 0}\frac1{i_1^{(l)}}\bigg)
    \prod_{l:k_l\ne 0}\Big(i_1^{(l)}\big(i_2^{(l)}-i_1^{(l)})\dots\big(i_{k_l}^{(l)}-i_{k_{l}-1}^{(l)}\big)\bigg)\Bigg]^{p/2}
\end{align*}
and
\begin{align*}
    \E &|\conv P_d\BS_{d}^0|^p=\prod_{i=1}^{d-1}\frac{\Gamma\big(\frac{i+p}{2}\big)}{\Gamma\big(\frac{i}{2}\big)}
    \\
    &\times\Bigg[\frac{2^{d-1}}{((d-1)!)^2}\prod_{\substack{l\,:\,k_l\ne 0\\ l\ne 0}}\bigg(i_1^{(l)}\big(i_2^{(l)}-i_1^{(l)})\dots\big(i_{k_l}^{(l)}-i_{k_{l}-1}^{(l)}\big)\bigg)\Bigg]^{p/2}.
\end{align*}
\end{proposition}
\begin{proof}
The first equation follows from~Corollary~\ref{937}, while the second one follows from Corollary~\ref{1733}, see Section~\ref{1138}.
\end{proof}

\subsection{Facets probabilities}
The first direct application of Theorem~\ref{2047} is a formula for the probability that $\conv\BS_{d}$ is a facet of ${\cQ_{d}}$.
\begin{theorem}\label{2048}
Consider any $\BS_{d}$ defined in~\eqref{1212} with some integers from~\eqref{1207} and indices from~\eqref{1208}. Then,
\begin{align*}
	 \notag\P&[\conv\BS_{d}\in\cF({\cQ_{d}})]=\frac{2}{\sqrt{2\pi}}\cdot\bigg(\sum_{k_l \ne0}\frac1{i_1^{(l)}}\bigg)^{1/2}
	 \\
	 &\hphantom{=}
	 \times\prod_{k_l \ne0}\Bigg[\frac1{i_2^{(l)}-i_1^{(l)}}\dots\frac1{i_{k_l}^{(l)}-i_{k_{l}-1}^{(l)}}\frac{(2(n_l-i^{(l)}_{k_l})-1)!!}{(2(n_l-i^{(l)}_{k_l}))!!}\Bigg]\notag
	 \\
	 &\hphantom{=}
	 \times\int_{-\infty}^\infty
	 \exp\bigg(-\frac {r^2}{2}\sum_{k_l \ne0}\frac{1}{i_1^{(l)}}\bigg)
	 \prod_{k_l=0}p_{n_l}(r)\prod_{k_l \ne0}q_{i_1^{(l)}}(r)  \dd r
\end{align*} 
and
\begin{align*}
	 \notag\P&[\conv\BS_{d}\in\cF({\cQ_{d}^0})]=\frac{2}{\sqrt{2\pi}}\cdot\bigg(\sum_{k_l \ne0}\frac1{i_1^{(l)}}\bigg)^{1/2}
	 \\
	 &\hphantom{=}
	 \times\prod_{k_l \ne0}\Bigg[\frac1{i_2^{(l)}-i_1^{(l)}}\dots\frac1{i_{k_l}^{(l)}-i_{k_{l}-1}^{(l)}}\frac{(2(n_l-i^{(l)}_{k_l})-1)!!}{(2(n_l-i^{(l)}_{k_l}))!!}\Bigg]\notag
	 \\
	 &\hphantom{=}
	 \times\int_{0}^\infty
	 \exp\bigg(-\frac {r^2}{2}\sum_{k_l \ne0}\frac{1}{i_1^{(l)}}\bigg)
	 \prod_{k_l=0}p_{n_l}(r)\prod_{k_l \ne0}q_{i_1^{(l)}}(r)  \dd r.
\end{align*}  
Also consider any $\BS_{d}^0$ defined in~\eqref{01212} with some integers from~\eqref{01207} and indices from~\eqref{01208}. Then,  
\begin{align}\label{2137}
	 \P&[\conv\BS_{d}^0\in\cF({\cQ_{d}^0})]
	 \\\notag
	 &=2\prod_{l=1}^m\Bigg[\frac1{i_1^{(l)}}\frac1{i_2^{(l)}-i_1^{(l)}}\dots\frac1{i_{k_l}^{(l)}-i_{k_{l}-1}^{(l)}}\frac{(2(n_l-i^{(l)}_{k_l})-1)!!}{(2(n_l-i^{(l)}_{k_l}))!!}\Bigg],
\end{align}     
with the convention that  the factor in the product equals ${(2n_l-1)!!}/{(2n_l)!!}$ for $k_l=0$. 
\end{theorem}
\begin{proof}
The theorem readily follows from Theorem~\ref{2047} by taking $g\equiv 1$ and applying  Proposition~\ref{1157} with $p=1$.
\end{proof}

\subsection{Average number of facets}
The next direct application gives the average number of facets of ${\cQ_{d}}$ and ${\cQ_{d}^0}$.

\begin{theorem}\label{2148}
We have
\begin{align*}
    \E|\cF({\cQ_{d}})|=\sum_{\substack{k_1+\dots+k_m=d \\ 0\leq k_l\leq n_l, \,l=1,\dots,m}}\quad \sum_{\substack{1\leq i_1^{(l)}<\dots< i_{k_l}^{(l)}\leq n_l \\ l=1,\dots,m\,:\, k_l>0}}\P&[\conv\BS_{d}\in\cF({\cQ_{d}})],
\end{align*}
and
\begin{align*}
    \E|\cF({\cQ_{d}^0})|=&\sum_{\substack{k_1+\dots+k_m=d \\ 0\leq k_l\leq n_l, \,l=1,\dots,m}}\quad \sum_{\substack{1\leq i_1^{(l)}<\dots< i_{k_l}^{(l)}\leq n_l \\ l=1,\dots,m\,:\, k_l>0}}\notag\P[\conv\BS_{d}\in\cF({\cQ_{d}^0})]
    \\\notag&+\sum_{\substack{k_1+\dots+k_m=d-1 \\ 0\leq k_l\leq n_l, \,l=1,\dots,m}}\quad \sum_{\substack{1\leq i_1^{(l)}<\dots< i_{k_l}^{(l)}\leq n_l \\ l=1,\dots,m\,:\, k_l>0}} \P[\conv\BS_{d}^0\in\cF({\cQ_{d}^0})],
\end{align*}
where the facets probabilities in the right-hand sides are calculated in Theorem~\ref{2048}.
\end{theorem}
\begin{proof}
The theorem readily follows from~\eqref{1151} and~\eqref{01151} with $g\equiv1$ by taking the expectation.
\end{proof}

\subsection{Facets containing the origin: distribution-free formula}
It follows from the proof of Theorem~\ref{2047} (see~\eqref{2134}) that~\eqref{2137} holds not only for the normally distributed steps, but for any symmetric continuously distributed ones. Thus we have the following distribution-free formula:
\begin{align*}
    \E|\cF^0({\cQ_{d}^0})|=2&\sum_{\substack{k_1+\dots+k_m=d-1 \\ 0\leq k_l\leq n_l, \,l=1,\dots,m}}\quad \sum_{\substack{1\leq i_1^{(l)}<\dots< i_{k_l}^{(l)}\leq n_l \\ l=1,\dots,m\,:\, k_l>0}}
    \\
    &\prod_{l=1}^m\Bigg[\frac1{i_1^{(l)}}\frac1{i_2^{(l)}-i_1^{(l)}}\dots\frac1{i_{k_l}^{(l)}-i_{k_{l}-1}^{(l)}}\frac{(2(n_l-i^{(l)}_{k_l})-1)!!}{(2(n_l-i^{(l)}_{k_l}))!!}\Bigg],
\end{align*}
where $\cF^0(\cdot)$ stands for the set of facets containing the origin as a vertex (empty if 0 is not a vertex).

\subsection{Expected  surface area}
In this subsection we derive  formulae for the expected surface areas (i.e.~$(d-$1$)$-dimensional content) of the boundaries of the convex hulls, $\bd{\cQ_{d}}$ and $\bd{\cQ_{d}^0}$. To this end, let us first apply Theorem~\ref{2047} with
\begin{align*}
    g(\bx_1,\dots,\bx_d)=|\conv(\bx_1,\dots,\bx_d)|
\end{align*}
and then Proposition~\ref{1157} with $p=2$ together with  the Legendre duplication formula
\begin{align}\label{2228}
     (d-1)!=\Gamma(d)=\frac{2^{d-1}}{\sqrt\pi}\Gamma\left(\frac{d}{2}\right)\Gamma\left(\frac{d+1}{2}\right).
\end{align}

We obtain that for any $\BS_{d}$ defined in~\eqref{1212} with some integers from~\eqref{1207} and indices from~\eqref{1208},
\begin{align}\label{855}
	 \E&\big[|\conv\BS_{d}|\ind_{\lbrace\conv\BS_{d}\in\cF({\cQ_{d}})\rbrace}\big]
	 = \frac{1}{2^{d/2-1}\Gamma\big(\frac d2\big)}
	 \sum_{k_l \ne0}\frac1{i_1^{(l)}}
	 \\\notag
	 &\hphantom{=}
	 \times\prod_{k_l \ne0}\Bigg[\sqrt{i_1^{(l)}}\cdot\frac1{\sqrt{i_2^{(l)}-i_1^{(l)}}}\dots\frac1{\sqrt{i_{k_l}^{(l)}-i_{k_{l}-1}^{(l)}}}\frac{(2(n_l-i^{(l)}_{k_l})-1)!!}{(2(n_l-i^{(l)}_{k_l}))!!}\Bigg]\notag
	 \\\notag
	 &\hphantom{=}
	 \times\int_{-\infty}^\infty
	 \exp\bigg(-\frac {r^2}{2}\sum_{k_l \ne0}\frac{1}{i_1^{(l)}}\bigg)
	 \prod_{k_l=0}p_{n_l}(r)\prod_{k_l \ne0}q_{i_1^{(l)}}(r)  \dd r
\end{align} 
and
\begin{align}\label{856}
	 \E&\big[|\conv\BS_{d}|\ind_{\lbrace\conv\BS_{d}\in\cF({\cQ_{d}^0})\rbrace}\big]
	 = \frac{1}{2^{d/2-1}\Gamma\big(\frac d2\big)}
	 \sum_{k_l \ne0}\frac1{i_1^{(l)}}
	 \\\notag
	 &\hphantom{=}
	 \times\prod_{k_l \ne0}\Bigg[\sqrt{i_1^{(l)}}\cdot\frac1{\sqrt{i_2^{(l)}-i_1^{(l)}}}\dots\frac1{\sqrt{i_{k_l}^{(l)}-i_{k_{l}-1}^{(l)}}}\frac{(2(n_l-i^{(l)}_{k_l})-1)!!}{(2(n_l-i^{(l)}_{k_l}))!!}\Bigg]\notag
	 \\\notag
	 &\hphantom{=}
	 \times\int_{0}^\infty
	 \exp\bigg(-\frac {r^2}{2}\sum_{k_l \ne0}\frac{1}{i_1^{(l)}}\bigg)
	 \prod_{k_l=0}p_{n_l}(r)\prod_{k_l \ne0}q_{i_1^{(l)}}(r)  \dd r,
\end{align} 
and also for any $\BS_{d}^0$ defined in~\eqref{01212} with some integers from~\eqref{01207} and indices from~\eqref{01208},  
\begin{align}\label{857}
	 \E&\big[|\conv\BS_{d}^0|\ind_{\lbrace\conv\BS_{d}^0\in\cF({\cQ_{d}^0})\rbrace}\big]= \frac{\sqrt{2\pi}}{2^{d/2-1}\Gamma\big(\frac d2\big)}
	 \\\notag
	 &\times\prod_{l=1}^m\Bigg[\frac1{\sqrt{i_1^{(l)}}}\frac1{\sqrt{i_2^{(l)}-i_1^{(l)}}}\dots\frac1{\sqrt{i_{k_l}^{(l)}-i_{k_{l}-1}^{(l)}}}\frac{(2(n_l-i^{(l)}_{k_l})-1)!!}{(2(n_l-i^{(l)}_{k_l}))!!}\Bigg],
\end{align}     
with the convention that  the factor in the product equals ${(2n_l-1)!!}/{(2n_l)!!}$ for $k_l=0$. 

Now we are ready to derive the desired formula.
\begin{theorem}\label{1221}
We have
\begin{align*}
    \E|\bd{\cQ_{d}}|=\sum_{\substack{k_1+\dots+k_m=d \\ 0\leq k_l\leq n_l, \,l=1,\dots,m}}\quad \sum_{\substack{1\leq i_1^{(l)}<\dots< i_{k_l}^{(l)}\leq n_l \\ l=1,\dots,m\,:\, k_l>0}}\E\big[|\conv\BS_{d}|\ind_{\lbrace\conv\BS_{d}\in\cF({\cQ_{d}})\rbrace}\big],
\end{align*}
and
\begin{align*}
    \E|\bd{\cQ_{d}^0}|=&\sum_{\substack{k_1+\dots+k_m=d \\ 0\leq k_l\leq n_l, \,l=1,\dots,m}}\quad \sum_{\substack{1\leq i_1^{(l)}<\dots< i_{k_l}^{(l)}\leq n_l \\ l=1,\dots,m\,:\, k_l>0}}\notag\E\big[|\conv\BS_{d}|\ind_{\lbrace\conv\BS_{d}\in\cF({\cQ_{d}^0})\rbrace}\big]
    \\\notag&+\sum_{\substack{k_1+\dots+k_m=d-1 \\ 0\leq k_l\leq n_l, \,l=1,\dots,m}}\quad \sum_{\substack{1\leq i_1^{(l)}<\dots< i_{k_l}^{(l)}\leq n_l \\ l=1,\dots,m\,:\, k_l>0}} \E\big[|\conv\BS_{d}^0|\ind_{\lbrace\conv\BS_{d}^0\in\cF({\cQ_{d}^0})\rbrace}\big],
\end{align*}
where the summands in the right-hand sides are calculated in~\eqref{855},~\eqref{856}, and~\eqref{857}.
\end{theorem}
\begin{proof}
The theorem is a direct corollary of~\eqref{1151} and~\eqref{01151}.
\end{proof}

\subsection{Expected volume and other intrinsic volumes}
Here, using the result of the previous section we calculate the expected volumes of ${\cQ_{d}}$ and ${\cQ_{d}^0}$.
\begin{theorem}\label{2214}
We have
\begin{align*}
    \E&|{\cQ_{d}}|=\frac{\Gamma\big(\frac{d+1}{2}\big)}{d\sqrt{\pi}\,\Gamma\big(\frac{d}{2}\big)}
    \\
    &\times\sum_{\substack{k_1+\dots+k_m=d+1 \\ 0\leq k_l\leq n_l, \,l=1,\dots,m}}\quad \sum_{\substack{1\leq i_1^{(l)}<\dots< i_{k_l}^{(l)}\leq n_l \\ l=1,\dots,m\,:\, k_l>0}}\E\big[|\conv\BS_{d+1}|\ind_{\lbrace\conv\BS_{d+1}\in\cF({\cQ_{d+1}})\rbrace}\big],
\end{align*}
and
\begin{align*}
    \E&|{\cQ_{d}^0}|=\frac{\Gamma\big(\frac{d+1}{2}\big)}{d\sqrt{\pi}\,\Gamma\big(\frac{d}{2}\big)}
    \\
    &\times\bigg(\sum_{\substack{k_1+\dots+k_m=d+1 \\ 0\leq k_l\leq n_l, \,l=1,\dots,m}}\quad \sum_{\substack{1\leq i_1^{(l)}<\dots< i_{k_l}^{(l)}\leq n_l \\ l=1,\dots,m\,:\, k_l>0}}\notag\E\big[|\conv\BS_{d+1}|\ind_{\lbrace\conv\BS_{d+1}\in\cF({\cQ_{d+1}^0})\rbrace}\big]
    \\\notag
    &\hphantom{\times\bigg(}+\sum_{\substack{k_1+\dots+k_m=d \\ 0\leq k_l\leq n_l, \,l=1,\dots,m}}\quad \sum_{\substack{1\leq i_1^{(l)}<\dots< i_{k_l}^{(l)}\leq n_l \\ l=1,\dots,m\,:\, k_l>0}} \E\big[|\conv\BS_{d+1}^0|\ind_{\lbrace\conv\BS_{d+1}^0\in\cF({\cQ_{d+1}^0})\rbrace}\big]\bigg),
\end{align*}
where the summands in the right-hand sides are calculated in~\eqref{855},~\eqref{856}, and~\eqref{857} with $d$ replaced by $d+1$.
\end{theorem}
\begin{proof}
It is well known that  the orthogonal projection of the standard Gaussian distribution onto a linear subspace is again the standard Gaussian distribution of corresponding dimension. Thus  for any fixed $\bu\in\mathbb S^d$,
\begin{align*}
    (P_{\bu^\perp}S_i^{(1)})_{i=1}^{n_1},\dots, (P_{\bu^\perp}S_i^{(m)})_{i=1}^{n_m}
\end{align*}
can be considered as a collection of independent Gaussian random walks in $\R^{d-1}$. In particular, we may think that
\begin{align*}
    P_{\bu^\perp}\,{\cQ_{d}}\eqdistr{\cQ_{d-1}}.
\end{align*}
Recalling the Cauchy surface area formula (see, e.g.,~\cite[Eq.~(6.12)]{SW08}) leads here to
\begin{align*}
    \E|\bd\,{\cQ_{d}}|=\E\frac{2\sqrt{\pi}\,\Gamma\big(\frac{d+1}{2}\big)}{\Gamma\big(\frac{d}{2}\big)}\int_{\mathbb S^{d-1}}|P_{\bu^\perp}\,{\cQ_{d}}|\,\mu(\mathrm{d}\bu)=\frac{2\sqrt{\pi}\,\Gamma\big(\frac{d+1}{2}\big)}{\Gamma\big(\frac{d}{2}\big)}\E|{\cQ_{d-1}}|,
\end{align*}

\begin{align*}
    \E|{\cQ_{d-1}}|=\frac{\Gamma\big(\frac{d}{2}\big)}{2\sqrt{\pi}\,\Gamma\big(\frac{d+1}{2}\big)}\E|\bd\,{\cQ_{d}}|,
\end{align*}
and applying Theorem~\ref{1221}  concludes the proof.
\end{proof}
Similarly, by means of the Kubota formula, it is possible to calculate all expected intrinsic volumes of ${\cQ_{d}}$ and ${\cQ_{d}^0}$. We skip the details here.

\section{Examples}\label{2129}

In general, the factor in the right-hand sides of~\eqref{2140} and~\eqref{3140}
\begin{align}\label{1249}
    I_1:=\int_{-\infty}^\infty \exp\bigg(-\frac {r^2}{2}\sum_{l\,:\, k_l\ne0}\frac{1}{i_1^{(l)}}\bigg)\prod_{l\,:\,k_l =0}p_{n_l}(r)\prod_{l\,:\,k_l \ne0}q_{i_1^{(l)}}(r)  \dd r,
    \\\notag
     I_2:=\int_{0}^\infty \exp\bigg(-\frac {r^2}{2}\sum_{l\,:\, k_l\ne0}\frac{1}{i_1^{(l)}}\bigg)\prod_{l\,:\,k_l =0}p_{n_l}(r)\prod_{l\,:\,k_l \ne0}q_{i_1^{(l)}}(r)  \dd r
\end{align}
cannot be simplified. However, there are several cases when they can.
\subsection{Single random walk} 
Suppose that $m=1$. Using the notation of Subsection~\ref{1205} we have
\begin{align*}
    \cQ_d=\conv(S_1,\dots,S_n),\quad\cQ_d^0=\conv(S_0,S_1,\dots,S_n).
\end{align*}
Up to translations, $\cQ_d$ with $n$ steps has the same distribution as $\cQ_d^0$ with $n-1$ steps, so we can consider $\cQ_d$ only.

As in~\eqref{1207}, fix $d$ indices 
\begin{align*}
    1\leq i_1<\dots<i_d\leq n
\end{align*}
and let
\begin{align*}
    \BS_{d}:=\bigg(S_{i_1},\dots, S_{i_d}\bigg),\quad  \conv\BS_{d}:=\conv\left\{S_{i_1},\dots, S_{i_d}\right\}.
\end{align*}
In this case, $I_1,I_2$ defined in~\eqref{1249} reduce to
\begin{align*}
    I_1&=\int_{-\infty}^\infty \exp\bigg(-\frac {r^2}{2i_1}\bigg)q_{i_1}(r)  \dd r=\sqrt{2\pi i_1}p_{i_1-1}(0)=\sqrt{2\pi i_1}\frac{(2i_1-3)!!}{(2i_1-2)!!},
    \\
    I_2&=\int_{0}^\infty \exp\bigg(-\frac {r^2}{2i_1}\bigg)q_{i_1}(r)  \dd r=\sqrt{2\pi i_1}p_{i_1}(0)=\sqrt{2\pi i_1}\frac{(2i_1-1)!!}{(2i_1)!!},
\end{align*}
where we used~\eqref{1907} and~\eqref{1350}, the law of total probability, and~\eqref{1900}. Thus, it follows from Theorem~\ref{2148} and Theorem~\ref{2048} that
\begin{align*}
	 \E |\cF(\cQ_{d})|
	 =2\sum_{1\leq i_1<\dots<i_d\leq n}
	 \frac1{i_2-i_1}\dots\frac1{i_d-i_{d-1}}\frac{(2(n-i_{d})-1)!!}{(2(n-i_d))!!}\frac{(2i_1-3)!!}{(2i_1-2)!!}.
\end{align*}   
After a change of variables
\begin{align*}
    j=i_1,  \,\, \, j_1=i_2-i_1,\,\,\,\dots, \,\, \,  j_{d-1}=i_d-i_{d-1}, 
\end{align*}
we obtain
\begin{align*}
	 \E |\cF(\cQ_{d})|&= 2\sum_{\substack{j_1+\dots+j_{d-1}\leq n-1 \\ j_1,\dots,j_{d-1}\geq1}} 
	 \frac1{j_1}\dots\frac1{j_{d-1}}
	 \\
	 &\times\sum_{j=1}^{n-j_1-\dots-j_{d-1}}
	\frac{(2(n-j_1-\dots-j_{d-1}-j)-1)!!}{(2(n-j_1-\dots-j_{d-1}-j))!!}\frac{(2j-3)!!}{(2j-2)!!}
	\\
	&= 2\sum_{\substack{j_1+\dots+j_{d-1}\leq n-1 \\ j_1,\dots,j_{d-1}\geq1}} 
	 \frac1{j_1\dots j_{d-1}},
\end{align*}   
where in the last step we used~\eqref{2011}, and thus we recover~\eqref{2238} with $n$ replaced by $n-1$.

In the same way, it is not difficult to show that in the case of a single random walk  Theorem~\ref{2214} turns to the formula
\begin{align*}
	 \E|\cQ_{d}|= \frac{1}{d!}\sum_{\substack{j_1+\dots+j_{d}\leq n-1 \\ j_1,\dots,j_{d}\geq1}} 
	 \frac1{\sqrt{j_1\dots j_{d}}},
\end{align*}  
see~\cite[Section~4.1]{VZ18}.

\subsection{Gaussian polytopes} 
Suppose that $n_1=\dots=n_m=1$. Using the notation of Subsection~\ref{1205} we have
\begin{align*}
    \cQ_d=\conv(X^{(1)},X^{(2)}\dots,X^{(m)}),\quad\cQ_d^0=\conv(0,X^{(1)},X^{(2)}\dots,X^{(m)}),
\end{align*}
which are known as the \emph{Gaussian polytope} and the \emph{Gaussian polytope with~$0$} resp.

In this case, due to~\eqref{3}, $I_1$ and $I_2$ (defined in~\eqref{1249}) turn to
\begin{align*}
    I_1=(2\pi)^{d/2}\int_{-\infty}^\infty\,\Phi^{m-d}(r)\,\varphi^{d}(r)\dd r,\quad I_2=(2\pi)^{d/2}\int_{0}^\infty\,\Phi^{m-d}(r)\,\varphi^{d}(r)\dd r,
\end{align*}
Applying Theorem~\ref{2148} readily gives the number of the averages facets for $\cQ_d$ and $\cQ_d^0$:
\begin{align*}
	 \E |\cF(\cQ_{d})| = 2\sqrt{d}\,(2\pi)^{\frac{d-1}{2}}\binom{m}{d}\,\int_{-\infty}^{+\infty}\,\Phi^{m-d}(r)\,\varphi^{d}(r)\dd r
\end{align*}   
and
\begin{align*}
	 \E |\cF(\cQ_{d}^0)| =\frac{\binom{m}{d-1}}{2^{m-d-1}}+ 2\sqrt{d}\,(2\pi)^{\frac{d-1}{2}}\binom{m}{d}\,\int_{0}^{+\infty}\,\Phi^{m-d}(r)\,\varphi^{d}(r)\dd r.
\end{align*}  
The first formula has been obtained in~\cite{hR70} and later generalized in~\cite{HMR04} to faces of all dimensions. The second formula seems to be new.

Similarly,  from Theorem~\ref{2214} we readily obtain the formulae
\begin{align*}
 \E&|{\cQ_{d}}|=\binom{m}{d+1}\,\frac{(d+1)\pi^{d/2}}{\Gamma\big(\frac{d}{2}+1\big)}\int_{-\infty}^{+\infty}\,\Phi^{m-d-1}(r)\,\varphi^{d+1}(r)\dd r
\end{align*}
and
\begin{align*}
    \E|{\cQ_{d}^0}|=
     \frac{\binom{m}{d}}{2^{m-\frac{d}{2}}\Gamma\left(\frac{d}{2}+1\right)}
    +\binom{m}{d+1}\,\frac{(d+1)\pi^{d/2}}{\Gamma\big(\frac{d}{2}+1\big)}\int_0^\infty\,\Phi^{m-d-1}(r)\,\varphi^{d+1}(r)\dd r,
\end{align*}
which have been recently obtained in~\cite{KZ19}.

\section{Proof of Theorem~\ref{2047}}\label{2139}
\subsection{Stochastic form of Blaschke--Petkantschin formula}
We  need to  integrate a non-negative measurable  function of $(d-1)$-tuples of points in $\R^d$. For some reasons, it is more convenient  to integrate this function first over the $(d-1)$-tuples of points in a fixed linear hyperplane and then integrate over the set of all hyperplanes. The corresponding transformation formula is known as the linear Blaschke--Petkantschin formula (see~ \cite[Theorem 7.2.1]{SW08}):
\begin{align}\label{951}
    &\int_{(\R^d)^{d-1}}{h(\bx_1,\ldots,\bx_{d-1})\dd\bx_1\ldots \dd\bx_{d-1}}
    \\&\quad=\frac{d!\kappa_d}{2}\int_{\mathbb S^{d-1}} \int_{(\bu^\perp)^{d-1}}h(\bx_1,\ldots,\bx_{d-1})|\conv(0,\bx_1,\ldots,\bx_{d-1})|\notag
    \\&\hphantom{\quad=\frac{d!\kappa_d}{2}\int_{\mathbb S^{d-1}} \int_{(\bu^\perp)^{d}}}\times\lambda_{\bu^\perp}(\dd\bx_1)\ldots\lambda_{\bu^\perp}(\dd\bx_{d-1}) \,\dd\mu(\dd\bu)\notag,
\end{align}
where we write $\kappa_d = \pi ^{d/2}/\Gamma (\frac{d}{2}+1)$ for the volume of the $d$-dimensional unit ball.

A similar affine version (see~\cite[Theorem 7.2.7]{SW08} combined with ~\cite[Theorem~13.2.12.]{SW08}) may be stated as follows:
\begin{align}\label{1817}
    &\int_{(\R^d)^{d}}{h(\bx_1,\ldots,\bx_d)\dd\bx_1\ldots \dd\bx_d}
    \\&\quad={d!\kappa_d}\int_{\mathbb S^{d-1}} \int_{0}^{\infty}\int_{(\bu^\perp+r\bu)^{d}}h(\bx_1,\ldots,\bx_d)|\conv(\bx_1,\ldots,\bx_d)|\notag
    \\&\hphantom{\quad=\frac{d!\kappa_d}{2}\int_{\mathbb S^{d-1}} \int_{0}^{\infty}\int_{(\bu^\perp+r\bu)^{d}}}\times\lambda_{\bu^\perp}(\dd\bx_1)\ldots\lambda_{\bu^\perp}(\dd\bx_d)\,\lambda(\dd r) \,\dd\mu(\dd\bu)\notag
    \\&\quad={d!\kappa_d} \int_{\mathbb S^{d-1}} \int_{0}^{\infty}\int_{(\bu^\perp)^{d}}h(\bx_1+r\bu,\ldots,\bx_d+r\bu)|\conv(\bx_1,\ldots,\bx_d)|\notag
    \\&\hphantom{\quad=\frac{d!\kappa_d}{2} \int_{\mathbb S^{d-1}} \int_{0}^{\infty}\int_{(\bu^\perp)^{d}}}\times\lambda_{\bu^\perp}(\dd\bx_1)\ldots\lambda_{\bu^\perp}(\dd\bx_d)\,\lambda(\dd r) \,\dd\mu(\dd\bu).\notag
\end{align}
Using the Legendre duplication formula (see~\eqref{2228} with $d-1$ replaced by $d$) it is convenient to modify the factor in the right-hand side as follows:
\begin{align}\label{2234}
    d!\kappa_d=2^d\pi^{\frac{d-1}{2}}\Gamma\left(\frac{d+1}{2}\right).
\end{align}
We will also use the following version of the law of total expectation: for any random vectors $X_1,\dots,X_n\in\R^d$ having a joint distribution density, any $k\in\{1,\dots, n\}$ and  any non-negative measurable functions $h_1:(\R^d)^k\to\R$ and $h_2:(\R^d)^n\to\R$ we have
\begin{align*}
     \E[h_1(X_1,\dots,&X_k)h_2(X_1,\dots,X_n)]=\\&\int_{(\R^d)^k}\E[h_2(X_1,\dots,X_n)\given(X_1,\dots,X_k)=(\bx_1,\dots,\bx_k)]
     \\&\hphantom{\int_{(\R^d)^k}}\times  h_1(\bx_1,\dots,\bx_k) f_{k}(\bx_1,\dots,\bx_k)\dd \bx_1\dots \dd \bx_k,
\end{align*}
where $f_k$ is the joint distribution density of $X_1,\dots,X_k$.

Combining this with~\eqref{951} for $k=d-1$ and with~\eqref{1817}  for $k=d$ (we assume that $n\geq d$) and using Fubini's theorem together with~\eqref{2234} leads to
\begin{align}\label{2124}
    &\E[h_1(X_1,\dots,X_{d-1})h_2(X_1,\dots,X_n)]=2^{d-1}\pi^{\frac{d-1}{2}}\Gamma\left(\frac{d+1}{2}\right)
    \\\notag&\times\int_{\mathbb S^{d-1}} \int_{(\bu^\perp)^{d-1}}\E[h_2(X_1,\dots,X_n)\given(X_1,\dots,X_{d-1})=(\bx_1,\dots,\bx_{d-1})]
     \\&\hphantom{\times\int_{\mathbb S^{d-1}} \int_{(\bu^\perp)^{d}}}\times  h_1(\bx_1,\dots,\bx_{d-1}) f_{d-1}(\bx_1,\dots,\bx_{d-1})|\conv(0,\bx_1,\ldots,\bx_{d-1})|\notag
     \\&\hphantom{\times\int_{\mathbb S^{d-1}} \int_{(\bu^\perp)^{d}}}
   \times\lambda_{\bu^\perp}(\dd\bx_1)\ldots\lambda_{\bu^\perp}(\dd\bx_{d-1}) \,\dd\mu(\dd\bu)\notag
\end{align}
and
\begin{align}\label{2236}
    &\E[h_1(X_1,\dots,X_{d})h_2(X_1,\dots,X_n)]=2^d\pi^{\frac{d-1}{2}}\Gamma\left(\frac{d+1}{2}\right) 
    \\\notag&\times\int_{\mathbb S^{d-1}} \int_{0}^{\infty}\int_{(\bu^\perp)^{d}}\E[h_2(X_1,\dots,X_n)\given(X_1,\dots,X_{d})=(\bx_1+r\bu,\dots,\bx_{d}+r\bu)]
     \\&\hphantom{\times\int_{\mathbb S^{d-1}} \int_{0}^{\infty}\int_{(\bu^\perp)^{d}}}\times  h_1(\bx_1+r\bu,\dots,\bx_{d}+r\bu) f_{d-1}(\bx_1+r\bu,\dots,\bx_{d}+r\bu)\notag
     \\&\hphantom{\times \int_{\mathbb S^{d-1}} \int_{0}^{\infty}\int_{(\bu^\perp)^{d}}}\times|\conv(\bx_1,\ldots,\bx_{d})|\lambda_{\bu^\perp}(\dd\bx_1)\ldots\lambda_{\bu^\perp}(\dd\bx_d)\,\lambda(\dd r) \,\dd\mu(\dd\bu).\notag
\end{align}

Now we are ready to prove the theorem. It is convenient to start with the proof of the last relation.

\subsection{Proof of~\eqref{1152}} It follows from~\eqref{2124} and  the translation invariance of $g$ that
\begin{align}\label{1026}
	 &\E[g(\BS_{d}^0)\ind_{\lbrace\conv\BS_{d}^0\,\in\,\cF(\cQ_d^0)\rbrace}]=2^{d-1}\pi^{\frac{d-1}{2}}\Gamma\left(\frac{d+1}{2}\right)
	 \\\notag	&\times\int_{\mathbb S^{d-1}} \int_{(\bu^\perp)^{d-1}}\P[\conv\BS_{d}^0\in\cF(\cQ_d^0)\given\BS_{d}^0=(0,\bx_1,\dots,\bx_{d-1})]\notag
	 \\&\hphantom{\times\int_{\mathbb S^{d-1}}\int_{(\bu^\perp)^{d-1}}}
	 \times  g(0,\bx_1,\dots,\bx_{d-1}) f_{\BS_{d}^0}(\bx_1,\dots,\bx_{d-1})|\conv(0,\bx_1,\ldots,\bx_{d-1})|\notag
	 \\\notag&\hphantom{\times\int_{\mathbb S^{d-1}}\int_{(\bu^\perp)^{d-1}}}
	 \times\lambda_{\bu^\perp}(\dd\bx_1)\ldots\lambda_{\bu^\perp}(\dd\bx_{d-1})\,\dd r \,\dd\mu(\dd\bu),
\end{align}
where $f_{\BS_{d}}^0$ is the joint density of $\BS_{d}^0$ (without its zero component), that is,
\begin{align}\label{1541}
    f_{\BS_{d}^0}&(\bx_1,\dots,\bx_{d-1})=(2\pi)^{-d(d-1)/2}\prod_{l\,:\,k_l\ne0} \big[i_1^{(l)}(i_2^{(l)}-i_1^{(l)}) \dots (i_{k_l}^{(l)}-i_{k_{l}-1}^{(l)})\big]^{-d/2}\notag
    \\&\times\exp\left(-\frac12\sum_{\substack{l\,:\,k_l\ne0\\ l>0}}\left[\frac{\Vert\bx_{1}^{(l)}\Vert^2}{i_1^{(l)}}+\frac{\Vert\bx_{2}^{(l)}-\bx_{1}^{(l)}\Vert^2}{i_2^{(l)}-i_1^{(l)}}+\dots+\frac{\Vert\bx_{k_l}^{(l)}-\bx_{k_{l}-1}^{(l)}\Vert^2}{i_{k_l}^{(l)}-i_{k_{l}-1}^{(l)}}\right]\right).
\end{align}
Here, to make the expression more compact, we used the notation
\begin{align*}
    \bx_{i}^{(l)}:=\bx_{k_1+\dots+k_{l-1}+i}\quad \text{\ for\ } l>0 \text{\ such that\ } k_l \ne 0.\\
\end{align*}
Let us calculate the probability under the integral in the right-hand side of~\eqref{1026}.
Assuming that $\conv\BS_{d}^0$ has non-zero $(d-$1$)$-dimensional content (and thus, $\Span\BS_{d}^0=\bu^\perp$), which holds with probability~$1$, we obtain that $\conv\BS_{d}^0\in\cF({\cQ_{d}^0})$ if and only if ${\cQ_{d}^0}$ lies in the closed half-space with boundary $\bu^\perp$. Equivalently, the values of the projections of all random walks onto $\Span\bu$ are simultaneously either non-negative or non-positive. Because of the symmetry, the latter two events have the same probability.  Thus,
\begin{align*}
	\P&[\conv\BS_{d}^0\in\cF(\cQ_d^0)\given\BS_{d}^0=(0,\bx_1,\dots,\bx_{d-1})]
	\\&=2\P\Big[P_{u}S_{1}^{(1)},\dots,P_{u}S_{n_m}^{(m)}\geq 0 \bgiven P_{u}S_{i^{(l)}_1}^{(l)}=\dots=P_{u}S_{i^{(l)}_{k_l}}^{(l)}=0\text{ for }l:k_l>0\Big]\notag
	\\\notag
	&=2\prod_{l=1}^{m}\P\Big[P_{u}S_{1}^{(l)},\dots,P_{u}S_{n_l}^{(l)}\geq 0 \bgiven P_{u}S_{i^{(l)}_1}^{(l)}=\dots=P_{u}S_{i^{(l)}_{k_l}}^{(l)}=0\Big].
\end{align*}
Here, we imply that if $k_l=0$, then the corresponding probability is unconditional, and therefore due to~\eqref{1141} equals
\begin{align*}
	\P\Big[P_{u}S_{1}^{(l)},\dots,P_{u}S_{n_l}^{(l)}\geq 0\Big]=\frac{(2n_l-1)!!}{(2n_l)!!}.
\end{align*}
Let us evaluate the conditional one. Since the increments of the random walks are independent, for  $l$ such that $k_l\ne0$,
\begin{align*}
	\P&\Big[P_{u}S_{1}^{(l)},\dots,P_{u}S_{n_l}^{(l)}\geq 0 \bgiven P_{u}S_{i^{(l)}_1}^{(l)}=\dots =P_{u}S_{i^{(l)}_{k_l}}^{(l)}=0\Big]
	\\&=\prod_{j=0}^{k_l-1}\P\Big[P_{u}S_{i^{(l)}_j+1}^{(l)},\dots,P_{u}S_{i^{(l)}_{j+1}-1}^{(l)}\geq 0 \bgiven P_{u}S_{i^{(l)}_j}^{(l)}=P_{u}S_{i^{(l)}_{j+1}}^{(l)}=0\Big]
	\\&\phantom{=}\times\P\Big[P_{u}S_{i^{(l)}_{k_l}+1}^{(l)},\dots,P_{u}S_{n_l}^{(l)}\geq 0 \bgiven P_{u} S_{i^{(l)}_{k_l}}^{(l)}=0\Big],
\end{align*}
where we write $i_0^{(l)}:=0$. It was proved by Sparre Andersen~\cite{eSA53} that the probability for a random walk bridge of length $q$ to stay non-negative (or equivalently non-positive) equals $1/q$. Thus,
\begin{align}\label{1114}
	\P\bigg[P_{u}S_{i^{(l)}_j+1}^{(l)},\dots,P_{u}S_{i^{(l)}_{j+1}-1}^{(l)}\geq 0 \given P_{u}S_{i^{(l)}_j}^{(l)}=P_{u}S^{(l)}_{i^{(l)}_{j+1}}=0\bigg]=\frac{1}{i^{(l)}_{j+1}-i^{(l)}_j}.
\end{align}
After time changing, the last part turns to the persistence probability of the symmetric random walk, so due to~\eqref{1141},
\begin{align}\label{1120}
    \P&\Big[P_{u}S_{i^{(l)}_{k_l}+1}^{(l)},\dots,P_{u}S_{n_l}^{(l)}\geq 0 \bgiven P_{u} S_{i^{(l)}_{k_l}}^{(l)}=0\Big]
    \\\notag
    &=\P\Big[P_{u}\Big(S_{i^{(l)}_{k_l}+1}^{(l)}-S_{i^{(l)}_{k_l}}^{(l)}\Big),\dots,P_{u}\Big(S_{n_l}^{(l)}-S_{i^{(l)}_{k_l}}^{(l)}\Big)\geq 0\Big]
    \\\notag
    &=\frac{\big(2\big(n_l-i^{(l)}_{k_l}\big)-1\big)!!}{\big(2\big(n_l-i^{(l)}_{k_l}\big)\big)!!}.
\end{align}
Combining the above we arrive at
\begin{align}\label{2134}
    \P&[\conv\BS_{d}^0\in\cF(\cQ_d^0)\given\BS_{d}^0=(0,\bx_1,\dots,\bx_{d-1})]
    \\
    &=2\prod_{\substack{l\,:\,k_l=0\\ l \ne 0}}\left[\frac{(2n_l-1)!!}{(2n_l)!!}\right]\prod_{\substack{l\,:\,k_l\ne0\\ l \ne 0}}\left[\frac{1}{i^{(l)}_1}\frac{1}{i^{(l)}_2-i^{(l)}_1}\dots\frac{1}{i^{(l)}_{k_l}-i^{(l)}_{k_{l}-1}}\frac{(2(n_l-i^{(l)}_{k_l})-1)!!}{(2(n_l-i^{(l)}_{k_l}))!!}\right].\notag
\end{align}
Applying this to~\eqref{1026} and using Fubini's theorem gives
\begin{align}\label{1149}
    \E[g(&\BS_{d}^0)\ind_{\lbrace\conv\BS_{d}^0\,\in\,\cF(\cQ_d)\rbrace}]=2^d\pi^{\frac{d-1}{2}}\Gamma\left(\frac{d+1}{2}\right) \prod_{\substack{l\,:\,k_l=0\\ l \ne 0}}\left[\frac{(2n_l-1)!!}{(2n_l)!!}\right]
	 \\\notag
	 & \times\prod_{\substack{l\,:\,k_l\ne0\\ l \ne 0}}\left[\frac{1}{i^{(l)}_1}\frac{1}{i^{(l)}_2-i^{(l)}_1}\dots\frac{1}{i^{(l)}_{k_l}-i^{(l)}_{k_{l}-1}}\frac{(2(n_l-i^{(l)}_{k_l})-1)!!}{(2(n_l-i^{(l)}_{k_l}))!!}\right]
	 \\\notag
	 &	 \times\int_{\mathbb S^{d-1}} \int_{(\bu^\perp)^{d-1}} g(0,\bx_1,\dots,\bx_{d-1}) f_{\BS_{d}^0}(\bx_1,\dots,\bx_{d-1})|\conv(0,\bx_1,\ldots,\bx_{d-1})|\notag
	 \\\notag
	 &\hphantom{\times\int_{\mathbb S^{d-1}} \int_{(\bu^\perp)^{d-1}}}\times\;\lambda_{\bu^\perp}(\dd\bx_1)\ldots\lambda_{\bu^\perp}(\dd\bx_{d-1})\,\dd r \,\dd\mu(\dd\bu).
\end{align}
\noindent Now, fix some $\bu\in\mathbb S^{d-1}$. Let $e_1,\dots, e_{d-1}$ be an orthonormal basis in~$\bu^\perp$. Let $Q$ be an orthogonal matrix with the columns  $e_1,\dots,e_{d-1}, u$:
\begin{align*}
    Q:=[e_1,\dots,e_{d-1}, u].
\end{align*}
Since $ Q^\top \bu=(0,\dots,0,1)$, changing the coordinates
\begin{align*}
    \bx_1,\dots,\bx_d\mapsto Q\bx_1,\dots,Q\bx_d
\end{align*}
leads to
\begin{align}\label{1955}
	 &\int_{(\bu^\perp)^{d-1}}g(0,\bx_1,\dots,\bx_{d-1}) f_{\BS_{d}}(\bx_1,\dots,\bx_{d-1})
	 \\\notag
	 &\hphantom{\int_{(\bu^\perp)^{d-1}}}
	 \times |\conv(0,\bx_1,\ldots,\bx_{d-1})|\lambda_{\bu^\perp}(\dd\bx_1)\ldots\lambda_{\bu^\perp}(\dd\bx_d)
	 \\\notag
	 &=\int_{(\R^{d-1})^{d-1}}g(0,\bx'_1,\dots,\bx'_{d-1}) f_{\BS_{d}}(\bx'_1,\dots,\bx'_d)|
	 \\\notag
	 &\hphantom{=\int_{(\R^{d-1})^{d-1}}}
	 \times|\conv(0,\bx_1,\ldots,\bx_{d-1})| \dd\bx_1\dots\dd\bx_{d-1},
\end{align}
where we used the rotational invariance of  the volume, $g$ and $f_{\BS_{d}}$, and 
\begin{align*}
    \text{for}\quad \bx=(x_1,\dots,x_{d-1})\quad\text{we write}\quad\bx':=(x_1,\dots,x_{d-1},0).
\end{align*}
It is well known that  the orthogonal projection of the standard Gaussian distribution onto a linear subspace is again the standard Gaussian distribution of corresponding dimension. Thus it follows from~\eqref{1541} (applied to dimension $d-1$ as well) that the function
\begin{align*}
    (2\pi)^{\frac{d-1}{2}}\prod_{\substack{l\,:\,k_l\ne 0\\ l \ne 0}} \big[i_1^{(l)}(i_2^{(l)}-i_1^{(l)}) \dots (i_{k_l}^{(l)}-i_{k_{l-1}}^{(l)})\big]^{1/2}f_{\BS_{d}}(\by'_1,\dots,\by'_{d-1}).
\end{align*}
coincides with the probability density function of $P_d\BS_{d}^0$, where 
\begin{align*}
    (P_dS_i^{(1)})_{i=1}^{n_1},\dots, (P_dS_i^{(m)})_{i=1}^{n_m}
\end{align*}
is considered as a collection of $m$ random walks in $\R^{d-1}$.
Therefore,
\begin{align*}
	&\E \big[g(P_d\BS_{d}^0)|\conv P_d\BS_{d}^0|\big]=(2\pi)^{\frac{d-1}{2}}\prod_{\substack{l\,:\,k_l\ne 0\\ l \ne 0}} \big[i_1^{(l)}(i_2^{(l)}-i_1^{(l)}) \dots (i_{k_l}^{(l)}-i_{k_{l-1}}^{(l)})\big]^{1/2}
	\\
	&\times\int_{(\R^{d-1})^{d-1}}g(0,\bx'_1,\dots,\bx'_{d-1}) f_{\BS_{d}}(0,\bx'_1,\dots,\bx'_{d-1})|\conv(0,\bx_1,\ldots,\bx_{d-1})| \dd\bx_1\dots\dd\bx_{d-1},
\end{align*}
which together with~\eqref{1955} leads to
\begin{align}\label{2022}
	 &\int_{(\bu^\perp)^{d-1}}g(0,\bx_1,\dots,\bx_{d-1}) f_{\BS_{d}^0}(\bx_1,\dots,\bx_{d-1})
	 \\\notag
	 &\hphantom{\int_{(\bu^\perp)^{d-1}}}
	 \times |\conv(0,\bx_1,\ldots,\bx_{d-1})|\lambda_{\bu^\perp}(\dd\bx_1)\ldots\lambda_{\bu^\perp}(\dd\bx_{d-1})
	 \\\notag
	 &=(2\pi)^{-\frac{d-1}{2}}\prod_{\substack{l\,:\,k_l\ne 0\\ l\ne 0}} \big[i_1^{(l)}(i_2^{(l)}-i_1^{(l)}) \dots (i_{k_l}^{(l)}-i_{k_{l-1}}^{(l)})\big]^{-1/2}\;\E \big[g(P_d\BS_{d}^0)|\conv P_d\BS_{d}^0|\big].
\end{align}
Integrating over $\mathbb S^{d-1}$ and combining with~\eqref{1149} gives~\eqref{1152}.

\subsection{Proof of~\eqref{3140}} It follows from~\eqref{2236} and  the translation invariance of $g$ that
\begin{align}\label{1105}
	 &\E[g(\BS_{d})\ind_{\lbrace\conv\BS_{d}\,\in\,\cF(\cQ_d^0)\rbrace}]=2^d\pi^{\frac{d-1}{2}}\Gamma\left(\frac{d+1}{2}\right) 
	 \\\notag	&\times\int_{\mathbb S^{d-1}} \int_{0}^{\infty}\int_{(\bu^\perp)^{d}}\P[\conv\BS_{d}\in\cF(\cQ_d^0)\given\BS_{d}=(\bx_1+r\bu,\dots,\bx_d+r\bu)]\notag
	 \\&\hphantom{\times\int_{\mathbb S^{d-1}} \int_{0}^{\infty}\int_{(\bu^\perp)^{d}}}
	 \times  g(\bx_1,\dots,\bx_d)\; f_{\BS_{d}}(\bx_1+r\bu,\dots,\bx_d+r\bu)\notag
	 \\\notag
	 &\hphantom{\times\int_{\mathbb S^{d-1}} \int_{0}^{\infty}\int_{(\bu^\perp)^{d}}}
	 \times|\conv(\bx_1,\ldots,\bx_d)|\;\lambda_{\bu^\perp}(\dd\bx_1)\ldots\lambda_{\bu^\perp}(\dd\bx_d)\,\dd r \,\dd\mu(\dd\bu),
\end{align}
where $f_{\BS_{d}}$ is the joint density of $\BS_{d}$, that is,
\begin{align}\label{1540}
    &f_{\BS_{d}}(\bx_1,\dots,\bx_d)=(2\pi)^{-d^2/2}\prod_{l\,:\,k_l\ne0} \big[i_1^{(l)}(i_2^{(l)}-i_1^{(l)}) \dots (i_{k_l}^{(l)}-i_{k_{l}-1}^{(l)})\big]^{-d/2}\notag
    \\&\times\exp\left(-\frac12\sum_{l\,:\,k_l\ne0}\left[\frac{\Vert\bx_{1}^{(l)}\Vert^2}{i_1^{(l)}}+\frac{\Vert\bx_{2}^{(l)}-\bx_{1}^{(l)}\Vert^2}{i_2^{(l)}-i_1^{(l)}}+\dots+\frac{\Vert\bx_{k_l}^{(l)}-\bx_{k_{l}-1}^{(l)}\Vert^2}{i_{k_l}^{(l)}-i_{k_{l}-1}^{(l)}}\right]\right).
\end{align}
As above, 
\begin{align*}
    \bx_{i}^{(l)}:=\bx_{k_1+\dots+k_{l-1}+i}\quad \text{\ for\ } l>0 \text{\ such that\ } k_l \ne 0.\\
\end{align*}
Let us calculate the probability under the integral in the right-hand side of~\eqref{1105}.
Assuming that $\conv\BS_{d}$ has non-zero $(d-$1$)$-dimensional content (and thus, $\aff\BS_{d}=\bu^\perp+r\bu$), which holds with probability~$1$, we obtain that $\conv\BS_{d}\in\cF({\cQ_{d}^0})$ if and only if ${\cQ_{d}^0}$ lies in the closed half-space with boundary $\bu^\perp+r\bu$ containing the origin. Equivalently, the values of the projections of all random walks onto $\Span\bu$ do not exceed~$r$. Thus,
\begin{align}\label{24}
	&\P[\conv\BS_{d}\in\cF(\cQ_d^0)\given\BS_{d}=(\bx_1+r\bu,\dots,\bx_d+r\bu)]\notag
	\\&=\P[P_{u}S_{1}^{(0)},P_{u}S_{1}^{(1)},\dots,P_{u}S_{n_m}^{(m)}\leq r \given P_{u}S_{i^{(l)}_1}^{(l)}=0,\dots,P_{u}S_{i^{(l)}_{k_l}}^{(l)}=0, l:k_l>0]\notag
	\\&=\prod_{l=1}^{m}\P[P_{u}S_{1}^{(0)},P_{u}S_{1}^{(l)},\dots,P_{u}S_{n_l}^{(l)}\leq r \given P_{u}S_{i^{(l)}_1}^{(l)}=r,\dots, P_{u}S_{i_{k^{(l)}_l}}^{(l)}=r].
\end{align}
Here, we imply that if $k_l=0$, then the corresponding probability is unconditional, and therefore equals
\begin{align*}
	\P[P_{u}S_{1}^{(l)},\dots,P_{u}S_{n_l}^{(l)}\leq r]=p_{n_l}(r).
\end{align*}
Let us evaluate the conditional one. Since the increments of the random walks are independent, for  $l$ such that $k_l\ne0$,
\begin{align*}
	\P&\left[P_{u}S_{1}^{(l)},\dots,P_{u}S_{n_l}^{(l)}\leq r \given P_{u}S_{i^{(l)}_1}^{(l)}=r,\dots, P_{u}S_{i^{(l)}_{k_l}}^{(l)}=r\right]
	\\&=\P\left[P_{u}S_{1}^{(l)},\dots,P_{u}S_{i^{(l)}_1-1}^{(l)}\leq r \given P_{u}S_{i^{(l)}_1}^{(l)}=r\right]
	\\&\phantom{=}\times\prod_{j=1}^{k_l-1}\P\left[P_{u}S_{i^{(l)}_j+1}^{(l)},\dots,P_{u}S_{i^{(l)}_{j+1}-1}^{(l)}\leq r \given P_{u}S_{i^{(l)}_j}^{(l)}=P_{u}S_{i^{(l)}_{j+1}}^{(l)}=r\right]
	\\&\phantom{=}\times\P\left[P_{u}S_{i^{(l)}_{k_l}+1}^{(l)},\dots,P_{u}S_{n_l}^{(l)}\leq r \given P_{u} S_{i^{(l)}_{k_l}}^{(l)}=r\right].
\end{align*}
Let us consider each of the three parts in the right-hand side separately. First, by definition,
\begin{align*}
	\P&\left[P_{u}S_{1}^{(l)},\dots,P_{u}S_{i^{(l)}_1-1}^{(l)}\leq r \given P_{u}S_{i^{(l)}_1}^{(l)}=r\right]=q_{i_1}(r).
\end{align*}
Second, as in~\eqref{1114},
\begin{align*}
	\P&\bigg[P_{u}S_{i^{(l)}_j+1}^{(l)},\dots,P_{u}S_{i^{(l)}_{j+1}-1}^{(l)}\leq r \given P_{u}S_{i^{(l)}_j}^{(l)}=P_{u}S^{(l)}_{i^{(l)}_{j+1}}=r\bigg]
	\\
	&=\P\bigg[P_{u}S_{i^{(l)}_j+1}^{(l)},\dots,P_{u}S_{i^{(l)}_{j+1}-1}^{(l)}\leq 0 \given P_{u}S_{i^{(l)}_j}^{(l)}=P_{u}S^{(l)}_{i^{(l)}_{j+1}}=0\bigg]
	\\
	&=\frac{1}{i^{(l)}_{j+1}-i^{(l)}_j}.
\end{align*}
Finally, as in~\eqref{1120},
\begin{align*}
    \P&\bigg[P_{u}S_{i^{(l)}_{k_l}+1}^{(l)},\dots,P_{u}S_{n_l}^{(l)}\leq r \given P_{u} S_{i^{(l)}_{k_l}}^{(l)}=r\bigg]
    \\&=\P\bigg[P_{u}\bigg(S_{i^{(l)}_{k_l}+1}^{(l)}-S_{i^{(l)}_{k_l}}^{(l)}\bigg),\dots,P_{u}\bigg(S_{n_l}^{(l)}-S_{i^{(l)}_{k_l}}^{(l)}\bigg)\leq 0\bigg]\\& =\frac{\big(2\big(n_l-i^{(l)}_{k_l}\big)-1\big)!!}{\big(2\big(n_l-i^{(l)}_{k_l}\big)\big)!!}.
\end{align*}
\noindent Inserting all the above equalities into (\ref{1105}), through (\ref{24}), leads to
\begin{align*}
	 &\E[g(\BS_{d})\ind_{\lbrace\conv\BS_{d}\in\cF({\cQ_{d}^0})\rbrace}]=2^d\pi^{\frac{d-1}{2}}\Gamma\left(\frac{d+1}{2}\right) 
	 \\
	 &\times\prod_{\substack{l\,:\,k_l\ne0\\ l \ne 0}}\left[\frac{1}{i^{(l)}_2-i^{(l)}_1}\dots\frac{1}{i^{(l)}_{k_l}-i^{(l)}_{k_{l}-1}}\frac{(2(n_l-i^{(l)}_{k_l})-1)!!}{(2(n_l-i^{(l)}_{k_l}))!!}\right]\notag
	 \\&\times\int_{\mathbb S^{d-1}} \int_{0}^{\infty}\int_{(\bu^\perp)^{d}}\bigg[\prod_{\substack{l\,:\,k_l=0\\ l \ne 0}}p_{n_l}(r)\bigg] \bigg[\prod_{\substack{l\,:\,k_l\ne0\\ l \ne 0}}q_{i_1^{(l)}}(r)\bigg]\notag g(\bx_1,\dots,\bx_d)
	 \\\notag
	 &\hphantom{\times \int_{\mathbb S^{d-1}} \int_{0}^{\infty}\int_{(\bu^\perp)^{d}}}
	 \times f_{\BS_{d}}(\bx_1+r\bu,\dots,\bx_d+r\bu) \cdot|\conv(\bx_1,\ldots,\bx_d)|
	 \\\notag
	 &\hphantom{\times \int_{\mathbb S^{d-1}} \int_{0}^{\infty}\int_{(\bu^\perp)^{d}}}
	 \times \lambda_{\bu^\perp}(\dd\bx_1)\ldots\lambda_{\bu^\perp}(\dd\bx_k)\,\dd r \,\dd\mu(\dd\bu).\notag
\end{align*}
Since $r\bu$ is orthogonal to $\bx_1,\dots,\bx_d$, it follows  from (\ref{1540}) that 
\begin{align}\label{1142}
    f_{\BS_{d}}(\bx_1+r\bu,\dots,\bx_d+r\bu)=\exp\left(-\frac{r^2}2\sum_{\substack{l\,:\,k_l\ne 0}}\frac{1}{i_1^{(l)}}\right) f_{\BS_{d}}(\bx_1,\dots,\bx_d).
\end{align}
Therefore,  
\begin{align}\label{2111}
	 &\E[g(\BS_{d})\ind_{\lbrace\conv\BS_{d}\in\cF({\cQ_{d}^0})\rbrace}]
	 =2^d\pi^{\frac{d-1}{2}}\Gamma\left(\frac{d+1}{2}\right) 
	 \\\notag
	 &\times\prod_{\substack{l\,:\,k_l\ne 0\\ l\ne 0}}\left[\frac{1}{i^{(l)}_2-i^{(l)}_1}\dots\frac{1}{i^{(l)}_{k_l}-i^{(l)}_{k_{l}-1}}\frac{(2(n_l-i^{(l)}_{k_l})-1)!!}{(2(n_l-i^{(l)}_{k_l}))!!}\right]
	 \\\notag
	 &\times\int_0^\infty\bigg[\prod_{\substack{l\,:\,k_l= 0\\ l \ne 0}}p_{n_l}(r)\bigg] \bigg[\prod_{\substack{l\,:\,k_l\ne 0}}q_{i_1^{(l)}}(r)\bigg] \exp\left(-\frac{r^2}2\sum_{\substack{l\,:\,k_l\ne 0}}\frac{1}{i_1^{(l)}}\right) \dd r
	 \\\notag
	 &\times\int_{\mathbb S^{d-1}} \int_{(\bu^\perp)^{d}}g(\bx_1,\dots,\bx_d) f_{\BS_{d}}(\bx_1,\dots,\bx_d) \\\notag
	 &\phantom{\times\int_{\mathbb S^{d-1}} \int_{(\bu^\perp)^{d}}}\times |\conv(\bx_1,\ldots,\bx_d)|\;\lambda_{\bu^\perp}(\dd\bx_1)\ldots\lambda_{\bu^\perp}(\dd\bx_d)\,\dd\mu(\dd\bu).
\end{align}
\medskip
Similarly to~\eqref{2022} we obtain 
\begin{align}\label{1156}
	 &\int_{(\bu^\perp)^{d-1}}g(\bx_1,\dots,\bx_d) f_{\BS_{d}}(\bx_1,\dots,\bx_d)
	 \\\notag
	 &\hphantom{\int_{(\bu^\perp)^{d-1}}}
	 \times |\conv(0,\bx_1,\ldots,\bx_{d})|\lambda_{\bu^\perp}(\dd\bx_1)\ldots\lambda_{\bu^\perp}(\dd\bx_d)
	 \\\notag
	 &=(2\pi)^{-d/2}\prod_{\substack{l\,:\,k_l\ne 0\\ l\ne 0}} \big[i_1^{(l)}(i_2^{(l)}-i_1^{(l)}) \dots (i_{k_l}^{(l)}-i_{k_{l-1}}^{(l)})\big]^{-1/2}\;\E \big[g(P_d\BS_{d})|\conv P_d\BS_{d}|\big].
\end{align}
Integrating over $\mathbb S^{d-1}$ and combining with~\eqref{2111} gives~\eqref{3140}.

\subsection{Proof of~\eqref{2140}}
As above, it follows from~\eqref{2236} and  the translation invariance of $g$ that
\begin{align*}
	 &\E[g(\BS_{d})\ind_{\lbrace\conv\BS_{d}\,\in\,\cF(\cQ_d)\rbrace}]=2^d\pi^{\frac{d-1}{2}}\Gamma\left(\frac{d+1}{2}\right) 
	 \\\notag	&\times\int_{\mathbb S^{d-1}} \int_{0}^{\infty}\int_{(\bu^\perp)^{d}}\P[\conv\BS_{d}\in\cF(\cQ_d)\given\BS_{d}=(\bx_1+r\bu,\dots,\bx_d+r\bu)]\notag
	 \\&\hphantom{\times\int_{\mathbb S^{d-1}} \int_{0}^{\infty}\int_{(\bu^\perp)^{d}}}
	 \times  g(\bx_1,\dots,\bx_d)\; f_{\BS_{d}}(\bx_1+r\bu,\dots,\bx_d+r\bu)\notag
	 \\\notag
	 &\hphantom{\times\int_{\mathbb S^{d-1}} \int_{0}^{\infty}\int_{(\bu^\perp)^{d}}}
	 \times|\conv(\bx_1,\ldots,\bx_d)|\;\lambda_{\bu^\perp}(\dd\bx_1)\ldots\lambda_{\bu^\perp}(\dd\bx_d)\,\dd r \,\dd\mu(\dd\bu),
\end{align*}
where $f_{\BS_{d}}$ is the joint density of $\BS_{d}$ defined in~\eqref{1540}.

Let us calculate the probability under the integral.
Assuming that $\conv\BS_{d}$ has non-zero $(d-$1$)$-dimensional content (and thus, $\aff\BS_{d}=\bu^\perp+r\bu$), which holds with probability~$1$, we obtain that $\conv\BS_{d}\in\cF({\cQ_{d}})$ if and only if ${\cQ_{d}}$ lies in one of the two closed half-spaces whose boundary coincides with $\bu^\perp+r\bu$. Equivalently, the values of the projections of all random walks forming $\cQ_d$ onto $\Span\bu$ either do not exceed $r$ or larger than $r$ simultaneously. Thus,
\begin{align}\label{1956}
	\P&[\conv\BS_{d}\in\cF(\cQ_d)\given\BS_{d}=(\bx_1+r\bu,\dots,\bx_d+r\bu)]
	\\\notag&=\P[P_{u}S_{1}^{(1)},\dots,P_{u}S_{n_m}^{(m)}\leq r \given P_{u}S_{i^{(1)}_1}^{(1)}=r,\dots,P_{u}S_{i^{(m)}_{k_m}}^{(m)}=r]
	\\\notag&\quad+\P[P_{u}S_{1}^{(1)},\dots,P_{u}S_{n_m}^{(m)}\geq r \given P_{u}S_{i^{(1)}_1}^{(1)}=r,\dots,P_{u}S_{i^{(m)}_{k_m}}^{(m)}=r]
	\\\notag&=\prod_{l=1}^{m}\P[P_{u}S_{1}^{(l)},\dots,P_{u}S_{n_l}^{(l)}\leq r \given P_{u}S_{i^{(l)}_1}^{(l)}=r,\dots, P_{u}S_{i^{(l)}_{k_l}}^{(l)}=r]
	\\\notag&\quad+\prod_{l=1}^{m}\P[P_{u}S_{1}^{(l)},\dots,P_{u}S_{n_l}^{(l)}\geq r \given P_{u}S_{i^{(l)}_1}^{(l)}=r,\dots, P_{u}S_{i^{(l)}_{k_l}}^{(l)}=r].
\end{align}
It is enough to consider the first summand, for the second one, due to the symmetry, can be obtained from it by replacing $r$ by $-r$. As before, we imply here that if $k_l=0$, then the corresponding probabilities are unconditional, and thus equal by definition to
\begin{align*}
	\P[P_{u}S_{1}^{(l)},\dots,P_{u}S_{n_l}^{(l)}\leq r]=p_{n_l}(r).
\end{align*}
Let us evaluate the conditional probabilities. Exactly as before, since the increments of the random walks are independent, for any $l$ such that $k_l\ne0$,
\begin{align*}
	\P&\left[P_{u}S_{1}^{(l)},\dots,P_{u}S_{n_l}^{(l)}\leq r \given P_{u}S_{i^{(l)}_1}^{(l)}=r,\dots, P_{u}S_{i^{(l)}_{k_l}}^{(l)}=r\right]\notag
	\\&=\P\left[P_{u}S_{1}^{(l)},\dots,P_{u}S_{i^{(l)}_1-1}^{(l)}\leq r \given P_{u}S_{i^{(l)}_1}^{(l)}=r\right]\notag
	\\&\phantom{=}\times\prod_{j=1}^{k_l-1}\P\left[P_{u}S_{i^{(l)}_j+1}^{(l)},\dots,P_{u}S_{i^{(l)}_{j+1}-1}^{(l)}\leq r \given P_{u}S_{i^{(l)}_j}^{(l)}=P_{u}S_{i^{(l)}_{j+1}}=r\right]\notag
	\\&\phantom{=}\times\P\left[P_{u}S_{i^{(l)}_{k_l}+1}^{(l)},\dots,P_{u}S_{n_l}^{(l)}\leq r \given P_{u} S_{i^{(l)}_{k_l}}^{(l)}=r\right].
\end{align*}
Let us again consider each of the three parts in the right-hand side separately.\\
For the first parts, by definition,
\begin{align*}
	\P&\left[P_{u}S_{1}^{(l)},\dots,P_{u}S_{i^{(l)}_1-1}^{(l)}\leq r \given P_{u}S_{i^{(l)}_1}^{(l)}=r\right]=q_{i_1}(r).
\end{align*}

\noindent Then for the second part, as in~\eqref{1114},
\begin{align*}
	\P\left[P_{u}S_{i^{(l)}_j+1}^{(l)},\dots,P_{u}S_{i^{(l)}_{j+1}-1}^{(l)}\leq r \given P_{u}S_{i^{(l)}_j}^{(l)}=P_{u}S^{(l)}_{i^{(l)}_{j+1}}=r\right]=\frac{1}{i^{(l)}_{j+1}-i^{(l)}_j}.
\end{align*}

\noindent Finally, for the third  part, as in~\eqref{1120},
\begin{align*}
    &\P\bigg[P_{u}S_{i^{(l)}_{k_l}+1}^{(l)},\dots,P_{u}S_{n_l}^{(l)}\leq r \given P_{u} S_{i^{(l)}_{k_l}}^{(l)}=r\bigg]
    \\&=\P\bigg[P_{u}\big(S_{i^{(l)}_{k_l}+1}^{(l)}-S_{i^{(l)}_{k_l}}^{(l)}\big),\dots,P_{u}\big(S_{n_l}^{(l)}-S_{i^{(l)}_{k_l}}^{(l)}\big)\geq 0\bigg]=\frac{\big(2n_l-i^{(l)}_{k_l})-1\big)!!}{\big(2(n_l-i^{(l)}_{k_l})\big)!!}.
\end{align*}
\medskip

\noindent Combining all these equalities gives an expression for the first summand in the right-hand side of~\eqref{1956}, and then changing $r$ to $-r$ gives an expression for the second one. Summing up,
\begin{align*}
	 &\E[g(\BS_{d})\ind_{\lbrace\conv\BS_{d}\in\cF({\cQ_{d}})\rbrace}]
	 =2^d\pi^{\frac{d-1}{2}}\Gamma\left(\frac{d+1}{2}\right)
	 \\
	 &\times\prod_{\substack{l\,:\,k_l\ne0\\ l \ne 0}}\left[\frac{1}{i^{(l)}_2-i^{(l)}_1}\dots\frac{1}{i^{(l)}_{k_l}-i^{(l)}_{k_{l}-1}}\frac{(2(n_l-i^{(l)}_{k_l})-1)!!}{(2(n_l-i^{(l)}_{k_l}))!!}\right]\notag
	 \\&\times\int_{\mathbb S^{d-1}} \int_{0}^{\infty}\int_{(\bu^\perp)^{d}}\bigg\lbrace\bigg[\prod_{\substack{l\,:\,k_l=0\\ l \ne 0}}p_{n_l}(r)\bigg] \bigg[\prod_{\substack{l\,:\,k_l\ne0\\ l \ne 0}}q_{i_1^{(l)}}(r)\bigg]\notag
	 +\bigg[\prod_{\substack{l\,:\,k_l=0\\ l \ne 0}}p_{n_l}(-r)\bigg] \bigg[\prod_{\substack{l\,:\,k_l\ne0\\ l \ne 0}}q_{i_1^{(l)}}(-r)\bigg]\bigg\rbrace
	 \\\notag
	 &\hphantom{\times \int_{\mathbb S^{d-1}} \int_{0}^{\infty}\int_{(\bu^\perp)^{d}}}
	 \times g(\bx_1,\dots,\bx_d) \,|\conv(\bx_1,\ldots,\bx_d)|\,
	 f_{\BS_{d}}(\bx_1+r\bu,\dots,\bx_d+r\bu) 
	 \\
	 &\hphantom{\times \int_{\mathbb S^{d-1}} \int_{0}^{\infty}\int_{(\bu^\perp)^{d}}}
	 \times\lambda_{\bu^\perp}(\dd\bx_1)\ldots\lambda_{\bu^\perp}(\dd\bx_k)\,\dd r \,\dd\mu(\dd\bu).\notag
\end{align*}
Using~\eqref{1142}, Fubini's theorem, and the formula
\begin{align*}
    \int_0^\infty (h(r)+h(-r))\dd r=\int_{-\infty}^\infty h(r)\dd r,\quad\text{where $h$ is even},
\end{align*}
leads to
\begin{align*}
	 \E&[g(\BS_{d})\ind_{\lbrace\conv\BS_{d}\in\cF({\cQ_{d}})\rbrace}]
	 =2^d\pi^{\frac{d-1}{2}}\Gamma\left(\frac{d+1}{2}\right) 
	 \\\notag
	 &\times\prod_{\substack{l\,:\,k_l\ne 0\\ l\ne 0}}\left[\frac{1}{i^{(l)}_2-i^{(l)}_1}\dots\frac{1}{i^{(l)}_{k_l}-i^{(l)}_{k_{l}-1}}\frac{(2(n_l-i^{(l)}_{k_l})-1)!!}{(2(n_l-i^{(l)}_{k_l}))!!}\right]
	 \\\notag
	 &\times\int_{-\infty}^\infty\bigg[\prod_{\substack{l\,:\,k_l= 0\\ l \ne 0}}p_{n_l}(r)\bigg] \bigg[\prod_{\substack{l\,:\,k_l\ne 0}}q_{i_1^{(l)}}(r)\bigg] \exp\left(-\frac{r^2}2\sum_{\substack{l\,:\,k_l\ne 0}}\frac{1}{i_1^{(l)}}\right) \dd r
	 \\\notag
	 &\times\int_{\mathbb S^{d-1}} \int_{(\bu^\perp)^{d}}g(\bx_1,\dots,\bx_d) f_{\BS_{d}}(\bx_1,\dots,\bx_d) \\\notag
	 &\phantom{\times\int_{\mathbb S^{d-1}} \int_{(\bu^\perp)^{d}}}\times |\conv(\bx_1,\ldots,\bx_d)|\;\lambda_{\bu^\perp}(\dd\bx_1)\ldots\lambda_{\bu^\perp}(\dd\bx_d)\,\dd\mu(\dd\bu).
\end{align*}
Finally, applying~\eqref{1156} and integrating over $\mathbb S^{d-1}$ gives~\eqref{2140}.

\section{Volumes of weighted Gaussian simplices.}\label{1138}
\subsection{Formulation of result}
Let
\begin{align*}
    \chi_1,\chi_2,\dots,\chi_d\in\R^1
\end{align*}
be independent random variables and suppose that for any $k=1,\dots,d,$  the random variable  $\chi_k$ has the chi distribution  with $k$ degrees of freedom (that is, the distribution of the norm of the $k$-dimensional standard Gaussian vector).

The main result of this section is the following formula for the volume moments of the weighted Gaussian simplex.
\begin{theorem}\label{840}
Fix some $l={1,\dots,d}$ and let $X_0,X_1,\dots, X_l\in\R^d$ be $d$-dimensional independent standard Gaussian vectors. Then for any $\sigma_0,\dots,\sigma_d>0$ and for any $p>0$,
\begin{align}\label{936}
    \E|\conv(\sigma_0 X_0,\dots, \sigma_l X_l)|^p
   =\left[\frac{2^{l/2}\sigma_0\dots\sigma_l}{l!}\sqrt{\frac{1}{\sigma_0^2}+\dots+\frac{1}{\sigma_l^2}}\,\right]^p
    \prod_{i=d-l+1}^{d}\frac{\Gamma\big(\frac{i+p}{2}\big)}{\Gamma\big(\frac{i}{2}\big)}.
\end{align}
\end{theorem}
The proof is given in the next subsection.
\begin{remark}
 It is well-known that
\begin{align}\label{2056}
    \E\chi_k^p=2^{p/2}\frac{\Gamma((k+p)/2)}{\Gamma(k/2)}.
\end{align}    
Therefore one might suggest that by the method of moments it follows from~\eqref{936} that
\begin{align}\label{942}
    |\conv(\sigma_0 X_0,\dots, \sigma_l X_l)|\eqdistr \frac{2^{l/2}}{l!}\sigma_0\dots\sigma_l\left(\frac{1}{\sigma_0^2}+\dots+\frac{1}{\sigma_l^2}\right)^{1/2}\,\chi_{d-l+1}\dots\chi_d.
\end{align}
In our paper, we do not need this distributional identity, so we skip the detailed proof. Moreover, we believe that it is possible to derive~\eqref{942} directly, without  calculating the moments.\footnote{Indeed, after we prepared the first version of this paper, it has been done in~\cite{tM20}.}
\end{remark}

Although Theorem~\ref{840} might be of  independent interest, the main reason why we need it is the following application to the convex hull of several weighted random walks with a total number of steps equal to~$d+1$.
\begin{corollary}\label{937}
Let $X_0,X_1,\dots, X_d\in\R^d$ be $d$-dimensional independent standard Gaussian vectors. 
Fix some $l={1,\dots,d}$ and indices $0=i_0<i_1<\dots<i_l\leq d$ and consider the following $l+1$ weighted random walks  defined as
\begin{align*}
    Y_{i_0}:=\sigma_0X_0, Y_1:=Y_{i_0}+\sigma_1X_1,&\dots,Y_{i_1-1}:=Y_{i_1-2}+\sigma_{i_1-1}X_{i_1-1},
    \\Y_{i_1}:=\sigma_{i_1}X_{i_1}, Y_{i_1+1}:=Y_{i_1}+\sigma_{i_1+1}X_{i_1+1},&\dots,Y_{i_2-1}:=Y_{i_2-2}+\sigma_{i_2-1}X_{i_2-1},
    \\&\dots,
    \\Y_{i_l}:=\sigma_{i_l}X_{i_l}, Y_{i_l+1}:=Y_{i_l}+\sigma_{i_l+1}X_{i_l+1},&\dots,Y_{d}:=Y_{d-1}+\sigma_{i_d}X_{d}.
\end{align*}
Then the $p$th moment of the volume of a random simplex with vertices at $Y_0,\dots,Y_d$ is given by
\begin{align}\label{1645}
    \E |\conv(Y_0,\dots,Y_d)|^p=\left[\frac{2^{d/2}\sigma_0\dots\sigma_{d}}{d!}\sqrt{\frac{1}{\sigma_{i_0}^2}+\dots+\frac{1}{\sigma_{i_l}^2}}\,\,\right]^p
    \prod_{i=1}^{d}\frac{\Gamma\big(\frac{i+p}{2}\big)}{\Gamma\big(\frac{i}{2}\big)}.
\end{align}
\end{corollary}
The proof is given in Subsection~\ref{1619}.
Letting $\sigma_0\to 0$ in~\eqref{1645} produces the following formula.
\begin{corollary}\label{1733}
Under the assumptions of Corollary~\ref{937} with $\sigma_0=0$ we have
\begin{align*}
    \E |\conv(0,Y_1,\dots,Y_d)|^p=\left[\frac{2^{d/2}\sigma_1\dots\sigma_d}{d!}\right]^p
    \prod_{i=1}^{d}\frac{\Gamma\big(\frac{i+p}{2}\big)}{\Gamma\big(\frac{i}{2}\big)}.
\end{align*}
\end{corollary}

\subsection{Proof of  Theorem~\ref{840}} 
Let us consider the case $l=d$ first, so the task is to show that
\begin{align}\label{1416}
    \E|\conv(\sigma_0 X_0,\dots, \sigma_d X_d)|^p
    =
    \left[\frac{2^{d/2}\sigma_0\dots\sigma_d}{d!}\sqrt{\frac{1}{\sigma_0^2}+\dots+\frac{1}{\sigma_d^2}}\,\right]^p
    \prod_{i=1}^{d}\frac{\Gamma\big(\frac{i+p}{2}\big)}{\Gamma\big(\frac{i}{2}\big)}.
\end{align}
Let $e_0,e_1,\dots,e_d$ be the standard orthonormal basis in $\R^{d+1}$. Consider a $d$-dimensional simplex $T\subset\R^{d+1}$ defined as
\begin{align*}
    T:=\conv(\sigma_0 e_0,\sigma_1 e_1,\dots,\sigma_d e_d).
\end{align*}

Denote by $A\in\R^{d\times(d+1)}$ a Gaussian matrix whose columns are $X_0,X_1,\dots, X_d$:
\begin{align*}
    A:=[X_0 X_1\dots X_d].
\end{align*}
It follows from~\cite[Proposition~4.1]{PP13} that
\begin{align}\label{2001}
    \E |AT|^p=\E{\det}^{p/2}(AA^\top)\int_{\mathbb S^{d}}|P_{\bu^\perp}T|^p\mu(\mathrm d\bu).
\end{align}
Note that $AT=\conv(\sigma_0X_0,\sigma_1 X_1,\dots, \sigma_d X_d)$ so the left-hand side of~\eqref{2001} coincides with the left-hand side of~\eqref{1416}. On the other hand, ${\det}^{p/2}(AA^\top)$ coincides with the volume of the parallelepiped spanned by the independent standard Gaussian vectors $Y_1,\dots,Y_d\in\R^{d+1}$ (corresponding to the rows of $A$). The moments of its volume are known (see, e.g.,~\cite[Theorem~2.2]{GKT17} or Corollary~\ref{2240} below):
\begin{align}\label{2138}
    \E{\det}^{p/2}(AA^\top)=2^{dp/2}\prod_{i=2}^{d+1}\frac{\Gamma\big(\frac{i+p}{2}\big)}{\Gamma\big(\frac{i}{2}\big)}.
\end{align}
There remains to calculate the integral in the right-hand side of~\eqref{2001}. Since $T$ is contained in an affine hyperplane orthogonal to
\begin{align*}
    \mathbf w:=\left(\frac{1}{\sqrt{d+1}},\dots,\frac{1}{\sqrt{d+1}}\right),
\end{align*}
we have
\begin{align}\label{2032}
    \int_{\mathbb S^{d}}|P_{\bu^\perp}T|^p\mu(\mathrm d\bu)=|T|^p \int_{\mathbb S^{d}}|\langle \bu, \mathbf w \rangle|^p\mu(\mathrm d\bu).
\end{align}
It is known~\cite[page~737]{HH08} that
\begin{align}\label{2145}
   |T|=\frac{\sigma_0\dots\sigma_d}{d!}\sqrt{\frac{1}{\sigma_0^2}+\dots+\frac{1}{\sigma_d^2}}.
\end{align}
To calculate the integral in the right-hand side of~\eqref{2032}, consider the standard Gaussian vector $Y\in\R^{d+1}$. It is well-known that it can be decomposed as
\begin{align*}
   Y=\frac{Y}{|Y|}\cdot |Y|\eqdistr U \cdot \chi_{d+1},
\end{align*}
where $U$ is uniformly distributed over $\mathbb S^{d}$ independently of $\chi_{d+1}$. Therefore,
\begin{align*}
   \E|\langle Y, \mathbf w \rangle|^p=\E \chi_{d+1}^p \int_{\mathbb S^{d}}|\langle \bu, \mathbf w \rangle|^p\mu(\mathrm d\bu).
\end{align*}
On the other hand, we know the moments of the standard Gaussian variable  $\langle Y, \mathbf w \rangle$:
\begin{align*}
   \E|\langle Y, \mathbf w \rangle|^p=\frac{2^{p/2}}{\sqrt\pi}\Gamma\left(\frac{p+1}{2}\right).
\end{align*}
Thus the latter two relations together  with~\eqref{2056} give
\begin{align*}
   \int_{\mathbb S^{d}}|\langle \bu, \mathbf w \rangle|^p\,\mu(\mathrm d d\bu)=\frac{\Gamma\left(\frac{d+1}{2}\right)\Gamma\left(\frac{p+1}{2}\right)}{\sqrt\pi\Gamma\left(\frac{d+1+p}{2}\right)},
\end{align*}
which together with~\eqref{2032} and~\eqref{2145} leads to
\begin{align*}
   \int_{\mathbb S^{d}}|P_{\bu^\perp}T|^p\mu(\mathrm d \bu)=\frac{\Gamma\left(\frac{d+1}{2}\right)\Gamma\left(\frac{p+1}{2}\right)}{\sqrt\pi (d!)^p\Gamma\left(\frac{d+1+p}{2}\right)}\sigma_0^p\dots\sigma_d^p\left(\frac{1}{\sigma_0^2}+\dots+\frac{1}{\sigma_d^2}\right)^{p/2}.
\end{align*}
Now combining this with~\eqref{2001} and~\eqref{2138} and using $\Gamma(1/2)=\sqrt\pi$ concludes the proof of~\eqref{1416}.

Now suppose that $l<d$. With probability one, 
\begin{align*}
   W_l:=\aff (\sigma_0 X_0,\dots, \sigma_l X_l)
\end{align*}   
is an affine  $l$-plane (which we will always assume). Denote by $O_{W_l}$ the orthogonal projection of the origin onto $W_l$. The rotational symmetry property of the standard Gaussian distribution implies that the \emph{linear} $l$-plane
\begin{align*}
    \tilde W_l:=W_L-O_{W_l}
\end{align*}
 is uniformly distributed over the $l$-dimensional Grassmannian with respect to the Haar measure independently of $|\conv(\sigma_0 X_0,\dots, \sigma_l X_l)|$. 
Let $P_l:\R^d\to\R^l$ denote the linear operator  of the orthogonal projection from $\R^d$ onto the first $l$ coordinates:
\begin{align*}
   P_l\,:\,\bx=(x_1,\dots,x_{d-1},x_d)\mapsto (x_1,\dots,x_{l}),
\end{align*}
and denote by $P_l^{\tilde W_l}:\tilde W_l\to\R^l$  its restriction to $\tilde W_l$.

We have
\begin{align*}
   |\conv(\sigma_1 P_lX_1,&\dots, \sigma_l P_lX_l)|
   \\
   &=|\conv(P_l^{\tilde W_l}(\sigma_0 X_1-O_{W_l}),\dots,  P_l^{\tilde W_l}(\sigma_lX_l-O_{W_l}))|
   \\
   &=|\conv(\sigma_0 X_1-O_{W_l},\dots,\sigma_lX_l-O_{W_l})|\cdot|\det P_l^{\tilde W_l}|
   \\
   &=|\conv(\sigma_0 X_1,\dots,\sigma_lX_l)|\cdot|\det P_l^{\tilde W_l}|.
\end{align*}   
With probability one, $\det P_l^{\tilde W_l}\ne 0$. Therefore,
\begin{align}\label{1145}
   \E|\conv&(\sigma_1 X_1,\dots, \sigma_l X_l)|^p=\E\left[\left(\frac{|\conv(\sigma_1 P_lX_1,\dots, \sigma_l P_lX_l)|}{|\det P_l^{W_l}|}\right)^p\,\right]
   \\\notag&=\E|\conv(\sigma_1 P_lX_1,\dots, \sigma_l P_lX_l)|^p\cdot\E|\det P_l^{W_l}|^{-p}
   \\\notag&= \left[\frac{2^{l/2}\sigma_0\dots\sigma_l}{l!}\sqrt{\frac{1}{\sigma_1^2}+\dots+\frac{1}{\sigma_l^2}}\,\right]^p
    \prod_{i=1}^{l}\frac{\Gamma\big(\frac{i+p}{2}\big)}{\Gamma\big(\frac{i}{2}\big)}\cdot\E|\det P_l^{\tilde W_l}|^{-p},
\end{align} 
where in the second step we used the independence of $\tilde W_l$ and \linebreak $|\conv(\sigma_1 X_1,\dots, \sigma_l X_l)|$ and in the last step we applied~\eqref{1416} with $d$ replaced by $l$. The latter was possible because, as was mentioned above, the orthogonal projection of the standard Gaussian distribution onto a linear subspace is again the standard Gaussian distribution of corresponding dimension.

There remains to compute $\E|\det P_l^{\tilde W_l}|^{-p}$. It is known (see, e.g.,~\cite[Theorem~2.3]{GKT17}) that
\begin{align*}
    \E&|\conv(X_0,\dots, X_l)|^p=\frac{(l+1)^{p/2}2^{pl/2}}{l!^p}\prod_{i=d-l+1}^{d}\frac{\Gamma\big(\frac{i+p}{2}\big)}{\Gamma\big(\frac{i}{2}\big)}.
\end{align*}
Comparing this and~\eqref{1145} with $\sigma_0=\dots=\sigma_l=1$ we obtain
\begin{align*}
    \E|\det P_l^{\tilde W_l}|^{-p}=\left.\prod_{i=d-l+1}^{d}\frac{\Gamma\big(\frac{i+p}{2}\big)}{\Gamma\big(\frac{i}{2}\big)}\middle/ \prod_{i=1}^{l}\frac{\Gamma\big(\frac{i+p}{2}\big)}{\Gamma\big(\frac{i}{2}\big)}\right.,
\end{align*}
and inserting it into~\eqref{1145} finishes the proof.

\subsection{Proof of Corollary~\ref{937}}\label{1619}
The following lemma must be known, however, we could not find a reference with the exact formulation, so for the reader's convenience we present a detailed proof.
\begin{lemma}\label{2055}
Fix some integers $k,l\geq0$ such that $k+l\leq d$ and let  $X_1,\dots, X_k\in\R^d$ be independent standard Gaussian vectors, while $Y_1,\dots,Y_l\in\R^d$ be some arbitrary random vectors independent of $X_1,\dots, X_k$. Then,
\begin{align*}
    |\conv(0,X_1,\dots, X_k,&Y_1,\dots,Y_l)|
    \\
    &\eqdistr\frac{l!}{(k+l)!} |\conv(0,Y_1,\dots,Y_l)|
    \prod_{i=d-l-k+1}^{d-l}\chi_{i},
\end{align*}
where the chi-distributed random variables are  independent of $(Y_1,\dots,Y_l)$.
\end{lemma}
\begin{proof}
The idea of the proof is well-known, see, e.g.,~\cite[Chapter~7]{tA03} or~\cite{NV14}. For $i=1,\dots,k$, denote by $W_{k-i+l}$ the linear span of $X_{i+1},\dots, X_k$ and $Y_1,\dots,Y_l$:
\begin{align*}
    W_{k-i+l}:=\Span(X_{i+1},\dots, X_k,Y_1,\dots,Y_l),
\end{align*}
and denote by $\dist(X_i,W_{k-i+l})$ the distance between $X_i$ and $W_{k-i+l}$.
By the ``base times height'' formula applied several times,
\begin{align*}
     |&\conv(0,X_1,\dots, X_k,Y_1,\dots,Y_l)|
     \\&=\frac{\dist(X_1,W_{k-1+l})}{k+l}\cdot|\conv(0,X_2,\dots, X_k,Y_1,\dots,Y_l)
     \\&=\frac{\dist(X_1,W_{k-1+l})}{k+l}\cdot\frac{\dist(X_2,W_{k-2+l})}{k-1+l}\cdot|\conv(0,X_3,\dots, X_k,Y_1,\dots,Y_l)|
     \\&=...
     \\&=\prod_{i=1}^k\frac{\dist(X_1,W_{k-i+l})}{k-i+1+l}\cdot|\conv(0,Y_1,\dots,Y_l)|
     \\&=\frac{l!}{(k+l)!}\prod_{i=1}^k\dist(X_i,W_{k-i+l})\cdot|\conv(0,Y_1,\dots,Y_l)|.
\end{align*}
Now fix some $i\in\{1,\dots,k\}$. Since the distribution of  $X_i$ is independent of $(X_{i+1},\dots, X_k,Y_1,\dots,Y_l)$ and spherically invariant, we have
\begin{align*}
    \dist(X_i,W_{k-i+l})\eqdistr\dist(X_i,E)=|P_{E^\perp} X_i|
\end{align*}
for any \emph{fixed} $(k-i+l)$-dimensional subspace $E$. Since $X_i\in\R^d$ is the standard Gaussian vector, $P_{E^\perp} X_i$ has the standard Gaussian distribution in $E\cong\R^{d-k+i-l}$, which means that $|P_{E^\perp} X_i|$ is chi-distributed with $d-k+i-l$ degrees of freedom, and the lemma follows.
\end{proof}
Integrating and applying~\eqref{2056} readily gives the following relation.
\begin{corollary}\label{35}
Under assumptions of Lemma~\ref{2055}, for any $p>0$ we have
\begin{align*}
    \E|&\conv(0,X_1,\dots, X_k,Y_1,\dots,Y_l)|^p
    \\&=\left[\frac{2^{k/2}l!}{(k+l)!}\right]^p\prod_{i=d-l-k+1}^{d-l}\frac{\Gamma\big(\frac{i+p}{2}\big)}{\Gamma\big(\frac{i}{2}\big)}\E|\conv(0,Y_1,\dots,Y_l)|^p.
\end{align*}
\end{corollary}
Taking $l=0$  gives the expression for the moments of the volume of the Gaussian simplex.
\begin{corollary}\label{2240}
Under assumptions of Lemma~\ref{2055}, for any $p>0$ we have
\begin{align*}
    \E|\conv(0,X_1,\dots,& X_k)|^p=\left[\frac{2^{k/2}}{k!}\right]^p\prod_{i=d-k+1}^{d}\frac{\Gamma\big(\frac{i+p}{2}\big)}{\Gamma\big(\frac{i}{2}\big)}.
\end{align*}
\end{corollary}

Now we are ready to prove Corollary~\ref{937}.

By the simplex volume formula,
\begin{align*}
   |\conv(Y_0,Y_1,\dots,Y_d)|&= |\conv(0,Y_1-Y_0,Y_2-Y_0,\dots,Y_d-Y_0)|
   \\&=\frac1{d!}|\det[Y_1-Y_0,Y_2-Y_0,\dots,Y_d-Y_0].
\end{align*}
Substracting  the $(i-1)$-th column from the $i$-th one subsequently for all $i=d,\dots,2,$ except for $i=i_l,\dots,i_1$, we obtain
\begin{align*}
   |\conv(Y_0,\dots,Y_d)|=\frac1{d!}|\det A|,
\end{align*}
where $A\in\R^{d\times d}$ is a matrix whose columns with the indices $i_2,\dots,i_l$ equal $\sigma_{i_2}X_{i_2}-\sigma_{0}X_0,\dots,\sigma_{i_l}X_{i_l}-\sigma_{0}X_0$, and for the remaining indices, the $i$-th column equals  $\sigma_{i}X_i$: 
\begin{align*}
   A:=[&\sigma_{1}X_1,\dots,\sigma_{i_2-1}X_{i_2-1},\sigma_{i_2}X_{i_2}-\sigma_{0}X_0,\sigma_{i_2+1}X_{i_2+1},
   \\&\dots, \sigma_{i_l+1}X_{i_l+1},\sigma_{i_l}X_{i_l}-\sigma_{i_0}X_0,\sigma_{i_l+1}X_{i_l+1},\dots,\sigma_{i_d}X_d].
\end{align*}
Thus,
\begin{align*}
   |\conv(Y_0,\dots,Y_d)|=|\conv(&0,\sigma_{1}X_1,\dots,\sigma_{i_2-1}X_{i_2-1},\sigma_{i_2}X_{i_2}-\sigma_{0}X_0,\sigma_{i_2+1}X_{i_2+1},
   \\&\dots, \sigma_{i_l+1}X_{i_l+1},\sigma_{i_l}X_{i_l}-\sigma_{i_0}X_0,\sigma_{i_l+1}X_{i_l+1},\dots,\sigma_{i_d}X_d)|.
\end{align*}
Applying Corollary~\ref{35} gives
\begin{align*}
   \E|&\conv(Y_0,\dots,Y_d)|^p=\left[2^{(d-l+1)/2}\frac{(l-1)!}{d!}\right]^p\prod_{i=0}^{d-l}\frac{\Gamma\big(\frac{d-l+1-i+p}{2}\big)}{\Gamma\big(\frac{d-l+1-i}{2}\big)}
   \\&\times\prod_{i\ne 0,i_2,\dots,i_l}\sigma_i\cdot\E|\conv(0,\sigma_{i_2}X_{i_2}-\sigma_{0}X_0,\dots,\sigma_{i_l}X_{i_l}-\sigma_{0}X_0)|^p.
\end{align*}
Now Corollary~\ref{2240} gives
\begin{align*}
   \E|\conv&(0,\sigma_{i_2}X_{i_2}-\sigma_{0}X_0,\dots,\sigma_{i_l}X_{i_l}-\sigma_{0}X_0)|^p
   \\&=\E|\conv(\sigma_{0}X_0,\sigma_{i_2}X_{i_2},\dots,\sigma_{i_l}X_{i_l})|^p
   \\&=\left[\frac{2^{(l-1)/2}}{(l-1)!}\right]^p\prod_{i=1}^{l-1}\frac{\Gamma\big(\frac{d-l+1+i+p}{2}\big)}{\Gamma\big(\frac{d-l+1+i}{2}\big)}\sigma_0^p\sigma_{1}^p\dots\sigma_{d}^p\left(\frac{1}{\sigma_0^2}+\frac{1}{\sigma_{i_2}^2}+\dots+\frac{1}{\sigma_{i_l}^2}\right)^{p/2}.
\end{align*}
Combining the latter two inequalities finishes the proof.

\section{Acknowledgments}

\noindent This work was supported by the Foundation for the Advancement of Theoretical Physics and Mathematics ``BASIS''. 

The work of DZ  funded by RFBR and DFG according to the research project  20- 51-12004.

JRFs research at Columbia University was supported by the Alliance Program.

\bibliographystyle{plain}
\bibliography{bib2}
\end{document}